\documentclass[11pt,a4paper]{article}

\usepackage[utf8]{inputenc}
\usepackage[T1]{fontenc}
\usepackage[margin=2.5cm]{geometry}
\usepackage{hyperref}
\usepackage{url}
\usepackage{amsmath,amssymb,amsthm,amsfonts,bbm}
\usepackage{dsfont}
\usepackage{graphicx}
\usepackage{microtype}
\usepackage{fancyhdr}

\newtheorem{theorem}{Theorem}[section]
\newtheorem{proposition}[theorem]{Proposition}

\newtheorem{definition}[theorem]{Definition}
\newtheorem{remark}[theorem]{Remark}

\title{Mild Solutions for Path-Dependent Parabolic PDEs with Neumann Boundary Conditions via Generalized BSDEs}

\author{
  Luca Di Persio\thanks{University of Verona, Italy. Email: \texttt{luca.dipersio@univr.it}} \and
  Matteo Garbelli\thanks{University of Verona, Italy. Email: \texttt{matteo.garbelli@univr.it}} \and
  Adrian Z\u{a}linescu\thanks{Alexandru Ioan Cuza University, Romania. Email: \texttt{adrian.zalinescu@info.uaic.ro}}
}

\begin{document}
\maketitle

\begin{abstract}
We study a system of Forward-Backward Stochastic Differential Equations (FBSDEs) with time-delayed generators. The forward process includes a reflection component expressed via a Stieltjes integral, while the backward process takes the form of a Generalized BSDE. We establish the connection between this FBSDE system and non-linear path-dependent PDEs with Neumann boundary conditions by deriving a representation formula.
\end{abstract}



\section{Introduction}

We aim to provide a probabilistic representation of a mild solution $u: [0,T] \times \Lambda \times \Gamma \rightarrow \mathbb{R}$
to the following nonlinear path-dependent Kolmogorov equation (PDKE) with
Neumann conditions
\begin{equation}
\left\{
\begin{array}{l}
\displaystyle -\partial _{t}u(t,\psi )-\mathcal{L}u(t,\psi )-f(t,\psi,u(t,\psi) , \partial
_{x}u\left( t,\psi \right) \sigma (t,\phi ),\left( u(\cdot ,\psi)\right)
_{t})=0 \, \smallskip,\\
\qquad\qquad\qquad\qquad\qquad\qquad\qquad\qquad\qquad\qquad\qquad\qquad\qquad\qquad\quad \psi =(\phi, \varphi) \in \Lambda \times \Gamma;\medskip  \\
\displaystyle \dfrac{\partial u}{\partial \nu}(t,\psi )+g(t,{\psi},u(t,\psi ),\left( u(\cdot
,\psi )\right) _{t})=0 \qquad  \textit{ if } \phi(t) \in \partial G  ;
\medskip  \\
\displaystyle u(T,\psi )=h(\psi ),\qquad \psi \in\Lambda \times \Gamma,
\end{array}%
\right.   \label{PDKE}
\end{equation}%
where $T<\infty $ is a fixed time horizon, $\bar{G}$ is the closure of a nonempty open convex bounded subset $G$ of $\mathbb{R}^d$, $\Lambda
:=\mathcal{C}([0,T]; \bar{G})$ is the space of continuous $\bar{G}
$-valued functions defined on $[0,T]$, while $\Gamma :=  \mathcal{C} ([0,T]; \mathbb{R}^d)  \cap \mathcal{BV}\left(  [0,T];\mathbb{R}^{d}\right) $ denotes the set of bounded variation and continuous $\mathbb{R}^d$-valued processes.

The space $\Gamma$ contains the \emph{reflection path} $K^{t,\psi}$, which represents the cumulative boundary correction needed to keep the forward process $X^{t,\psi}$ inside $\overline{G}$. The initial condition $\phi \in \Gamma$
specifies the history of this reflection up to the starting time $t$. Throughout the paper, we use the notation $\psi = (\phi, \varphi)$ to denote a couple in $\Lambda \times \Gamma$.

We denote by $\mathcal{L}$ the second-order differential operator given by
\begin{equation}
\label{operator}
\mathcal{L} u (t,\psi) := \frac{1}{2} \operatorname{Tr} \left[  \sigma (t,\phi)\sigma^{\ast} (t,\phi )\partial_{xx}^{2}
u(t,\psi ) \right] + \langle b(t,\phi) ,\partial_{x}u(t,\psi )\rangle ,
\end{equation}
with $b:\left[ 0,T\right] \times \Lambda \rightarrow \mathbb{R}^{d}$ and $\sigma :\left[ 0,T\right] \times \Lambda \rightarrow \mathbb{R}^{d\times d^{\prime }}$.

We define  the normal derivative of $u: [0,T] \times \Lambda \rightarrow \mathbb{R}$ appearing in the Neumann boundary condition by
\begin{equation}
\label{eq: directional}
\frac{\partial u}{\partial \nu}(t,\psi ) = \partial_x u (t,\psi ) \cdot \nu ( \phi (t)),
\end{equation}
where $\nu(x)$ is the interior unit normal to $\bar{G}$ in  $x\in\partial G$.

\noindent
Let $\delta>0$ be a fixed \textit{delay}. For a function $v:[0,T] \rightarrow \mathbb{R}$, we define its \textit{delayed path} as
\begin{equation}
\label{delayed_term}
v_t := v ((t + \theta)^+)_{\theta \in [-\delta, 0]} \, .
\end{equation}

We prove that under appropriate assumptions on the coefficients specified in Sec. \ref{sec:ass},  the
deterministic non-anticipative functional $u:\left[ 0,T\right] \times
\Lambda \rightarrow \mathbb{R}$ given by the representation formula%
\begin{equation}
    \label{fk}
u( t,\psi) :=Y^{t,\psi}\left( t\right) \, ,
\end{equation}
is a solution to equation (\ref{PDKE}), where $(X^{t,\psi} ,K^{t,\psi},Y^{t,\psi}
,Z^{t,\psi} )$
is the unique solution of a decoupled Forward Backward Stochastic
Differential Equation (FBSDE). The FSBDE system consists of a path-dependent forward reflected SDE for processes $(X^{t,\psi}, K^{t,\psi})$ and a backward component in the form of a generalized BSDE with generators depending on both the forward solution and time-delayed terms.

We assume initial data $\left( t,\psi \right) \in \left[ 0,T\right] \times \Lambda \times \Gamma$ and we introduce the following equation%
\begin{equation}
\left\{
\begin{array}{l}
\displaystyle X^{t,\psi}\left( s\right) =\phi \left( t\right) - \varphi \left( t\right)
+\int_{t}^{s}b(r,X^{t,\psi})dr+\int_{t}^{s}\sigma (r,X^{t,\psi})dW(
r) + K^{t,\psi}(s) ,\medskip  \\ \displaystyle
K^{t,\psi}(s)= \varphi(t) + \int_{t}^{s}\nu(X^{t,\psi}(r))\mathds{1}_{\left\{ X^{t,\psi
}\left( r\right) \in \partial G\right\} }dA^{t,\psi}(r)\, ,
\qquad s\in[t,T]\medskip  \\
\displaystyle \left(X^{t,\psi}\left( s\right), K^{t,\psi}\left( s\right) \right)
  = (\phi(s),\varphi(s))  \,
, \qquad s\in \lbrack 0,t)\bigskip  \, , \\
\displaystyle Y^{t,\psi}\left( s\right) =h({X}^{t,\psi}, {K}^{t,\psi}
)+\int_{s}^{T}f(r,{X}^{t,\psi}, {K}^{t,\psi},
Y^{t,\psi}\left( r\right) ,Z^{t,\psi}\left(
r\right) ,Y_{r}^{t,\psi})dr \smallskip  \\
\qquad  \qquad \displaystyle +\int_{s}^{T}g(r,{X}^{t,\psi}, {K}^{t,\psi},
Y^{t,\psi}\left( r\right) ,Y_{r}^{t,\psi})dA^{t,\psi}(r) \smallskip \\
\qquad  \qquad \displaystyle -\int_{s}^{T}Z^{t,\psi}\left( r\right) dW( r) , \qquad s\in[t,T]\medskip  \\
\displaystyle Y^{t,\psi}\left( s\right) =Y^{s,\psi}\left( s\right), \qquad s\in \lbrack 0,t) \, ,
\end{array}%
\right.   \label{FBSDE}
\end{equation}%
where $W$ is a standard Brownian motion and $dA^{t , \psi}$ denotes the total variation of $dK^{t,\psi}$. The notation $Y_r^{t,\psi}$ appearing in the generators $f$ and $g$ stands for the delayed path of the process $Y^{t,\psi}$, introduced in \eqref{delayed_term}.

The primary aim of this paper is to establish a novel probabilistic representation of mild solutions to nonlinear PDKE with Neumann boundary conditions \eqref{PDKE} in terms of the solution of the FBSDE \eqref{FBSDE}, which is characterized by comprehensive path-dependence. First, the PDE and BSDE coefficients depend on the entire path of the forward process. Second, the system features explicit time delays: the PDE's source term depends on the history of its own solution, $(u(\cdot,\psi))_t$, and the BSDE's generator relies on the history of the backward process, $Y_r^{t,\psi}$. This structure ensures the system's current dynamics are intrinsically linked to their past evolution.

At the heart of our problem is the reflection mechanism implemented via Stieltjes integrals concerning the boundary local time, enabling the modeling of constrained processes within the domain $\bar{G}$. This reflection appears in both the forward component, where the process $X^{t,\psi}$ is constrained to remain in $\bar{G}$ through the boundary regulator $K^{t,\psi}$, and in the backward component through the term $\int g\,dA^{t,\psi}$, which captures the boundary contributions tied to the Neumann conditions of the PDE.

A significant novelty of our work lies in the detailed analysis of the associated semigroup, which allows us to define a mild solution in this context rigorously. Our semigroup-based approach simultaneously accounts for path-dependence across the entire solution path, reflection at the boundary through Neumann conditions, and time-delay aspects in the PDE and its probabilistic representation. To the best of our knowledge, this comprehensive treatment has not been previously explored in this setting.

Pardoux and Zhang proved in \cite{pa-zh/98} the existence of a probabilistic representation for the viscosity solution of a Neumann boundary PDE by using a class of BSDEs involving a Stieltjes integral concerning the continuous increasing process. In this article, we generalize this classical result for a path-dependent Neumann PDE associated with a time-delayed BSDE.
A previous result of this type without considering path-dependent coefficients is also stated in
\cite{pa-ra/17}.  In \cite{zalin}, the authors state an approximation result for the viscosity solution of a system
of semi-linear Neumann PDE  with continuous coefficients. The approximation works with convergence in probability through a penalization method similar to the one we use for this article. In \cite{wong}, they prove a result of the existence and uniqueness of the Neumann boundary problem of a parabolic partial differential equation. Unlike our setting, the authors consider a singular nonlinear divergence term and study the PDE in a weak sense. In \cite{zalinescu}, the author works with
finite-dimensional multivalued SDEs in a weak setting in a more general framework still fitting the one analyzed in this paper. In \cite{brzez}, the authors consider a real-valued stochastic partial differential
equation (SPDEs) with reflection over an open unit ball in a separable Hilbert space. The initial value problem is described in terms of stochastic evolution equations where Riemann-Stieltjes integrals have Hilbert-valued coefficients.

The paper is structured as follows. In Section \ref{sec:ass}, we introduce the theoretical framework, specifying the assumptions on the coefficients and setting up the necessary background for the FBSDEs. Section \ref{sec:forward} focuses on the forward SDE with reflection and discusses the solution's existence, uniqueness, and Markovian properties. Section \ref{sec:fbsde} addresses the backward component and provides the key Feynman-Kac representation formula. In Section \ref{sec:gdg}, we introduce the concept of the generalized directional gradient and its relation to the solution of the BSDE, concluding with a penalization approach to solving the PDE. Finally, Section \ref{sec:mild} provides a formal definition of mild solutions in the context of path-dependent PDEs, incorporating both the semigroup structure and the contributions of the boundary conditions.

\section{Theoretical Framework}
\label{sec:ass}

Let $T>0$ be a finite horizon of time and
$\delta\in(0,T]$ a fixed time-delay. Let $G$ be an open, convex, bounded subset in $\mathbb{R}^d$. We assume the same assumptions as in \cite{pa-zh/98}, namely the existence of a
function $\ell \in \mathcal{C}^2_b (\mathbb{R}^d)$ parameterizing
$G = \{\ell> 0  \}$ and its boundary $\partial G = \{\ell= 0  \}$. We define the normal to $G$ for all $x \in \partial G$ as $\nu (x):=\nabla \ell(x)$ and assume that $|\nu(x)| = 1$, $\forall x\in\partial G$.

For $m \geq 1$ and a function $\phi$ in $\mathcal{C} ([0,T], \mathbb{R}^{m})$, we define $\Vert \phi \Vert$ as the uniform norm (supremum norm) defined by
$$
\Vert \phi \Vert = \sup_{s \in [0,T]}|\phi(s)| \, .
$$

We denote by \( \Lambda \) the subspace of continuous functions from the interval \([0, T]\) to \( G \):
\begin{equation}
\label{eq:lambda}
\Lambda := \mathcal{C}\left([0, T]; \bar{G}\right) = \left\{ \phi: [0, T] \rightarrow \bar{G} \ \big| \ \phi \text{ is continuous} \right\}.
\end{equation}

 By \( \Gamma \) we denote the space of continuous functions from the interval \([0, T]\) to \( \mathbb{R}^d \) that are of bounded variation
\begin{equation}
\label{eq:gamma}
\Gamma := \mathcal{C} ([0,T]; \mathbb{R}^d)  \cap \mathcal{BV}\left(  [0,T];\mathbb{R}^{d}\right) .
\end{equation}

For \( \varphi \in  \Gamma \), we define $ \Vert \varphi \Vert _{t,s}$ as the total variation of \( \varphi \) over the interval \( [t,s] \).

\noindent
On $\Lambda$,  $\Gamma$ and $\Lambda \times \Gamma$, we use the corresponding supremum norms unless otherwise specified.

\medskip

Let $\left(  \Omega,\mathcal{F},\mathbb{P}\right)  $ be a complete probability space, $W$ a $d^\prime$-dimensional
Brownian motion defined on $\left(  \Omega,\mathcal{F},\mathbb{P}\right)  $ and $ \mathcal{F}_s^t $ the $\sigma$-algebra generated by the increments of the Brownian motion $W (r \vee t) - W (t)$ with $r \in [0,s]$ and augmented by all $\mathbb{P}$-null sets.
Throughout the paper, we use the filtration $ \mathbb{F}^t = \{ \mathcal{F}_s^t \}_{s \in [0,T]}$ to emphasize the explicit dependence on an arbitrary initial time $t$. In the case $t=0$, we will drop the superscript $0$.

To ensure that the deterministic non-anticipative functional $u$, given by the representation formula \eqref{fk} is a solution to the nonlinear PDKE \eqref{PDKE}, the coefficients and functions involved in the PDKE and the decoupled FBSDE \eqref{FBSDE} are required to satisfy the following conditions.

 We suppose that there exist constants $L$, $M$, $\tilde{M}$, $L$, $\tilde{L}>0$, bounded measurable functions $L_1, \tilde{L}_1: \Gamma  \rightarrow\mathbb{R}_{+}$ and $\rho$, $\tilde{\rho}$ probability
measures on $ \mathcal{B}\left(  \left[  -\delta,0\right]
\right) $ such that:

\begin{itemize}

\item[$\mathrm{(A}_{1}\mathrm{)}$]  The functions $b: [0, T] \times \Lambda \to \mathbb{R}^d$ and $\sigma: [0, T] \times \Lambda \to \mathbb{R}^{d \times d'}$ are bounded and satisfy for all $t \in [0,T]$ and  $\phi_1,\phi_2 \in \Lambda$,
    \begin{equation}
    \label{ass1}
    |b(t, \phi_1) - b(t, \phi_2)| \leq L ||\phi_1 - \phi_2||,  \qquad
    |\sigma(t, \phi_1) - \sigma(t, \phi_2)| \leq L ||\phi_1 - \phi_2||.
    \end{equation}

\item[$\mathrm{(A}_{2}\mathrm{)}$] The function $h:\Lambda \times \Gamma \to\mathbb{R}$  is continuous and $|h(\psi )|\leq M(1+\left\Vert \psi \right\Vert^{p} ),$ for all $\psi \in {\Lambda \times \Gamma}.$

\item[$\mathrm{(A}_{3}\mathrm{)}$]  The function $f:[0,T] \times \Lambda \times \Gamma \times \mathbb{R} \times \mathbb{R}^{d^\prime} \times L^{2}\left( [-\delta ,0];\mathbb{R} \right) \rightarrow \mathbb{R}$ satisfies the following:
\begin{equation*}
\begin{array}{rl}
\left( i\right) & \displaystyle (\phi, \varphi) \mapsto f\left(t,\phi, \varphi ,y,z, \hat{y}\right) \
\text{ is continuous,}\smallskip \\
\left( ii\right) & \displaystyle|f(t,\phi ,\varphi,y,z,\hat{y})-f(t,\phi, \varphi
,y^{\prime },z^{\prime },\hat{y})|\leq L \left( |y-y^{\prime
}|+|z-z^{\prime }|\right),\smallskip \\
\multicolumn{1}{l}{\left( iii\right)} & \displaystyle|f(t,\phi ,\varphi,y,z,\hat{y})-f(t,\phi,\varphi,y,z,\hat{y}^{\prime })|^{2} \leq L_1 (\varphi)\int_{-\delta }^{0}
\left\vert \hat{y}(\theta )-\hat{y}^{\prime }(\theta )\right\vert
^{2} \rho (d\theta ),\smallskip \\
\left( iv\right) & \displaystyle\left\vert f\left( t,\phi, \varphi,0,0,0\right)
\right\vert \leq M(1+\left\Vert \phi \right\Vert + \Vert \varphi \Vert)
\end{array}%
\end{equation*}
for any $t\in \lbrack
0,T]$, $\phi \in {\Lambda }$, $\varphi \in \Gamma$, $\left( y,z\right) ,(y^{\prime
},z^{\prime })\in \mathbb{R} \times \mathbb{R}^{d'}$ and $\hat{y},\hat{y}^{\prime }\in L^{2}\left( [-\delta ,0];\mathbb{R} \right) $.

\item[$\mathrm{(A}_{4}\mathrm{)}$] The function $g:[0,T] \times \Lambda \times \Gamma \times \mathbb{R}  \times L^{2}\left( [-\delta ,0];\mathbb{R} \right) \rightarrow \mathbb{R}$ satisfies the following:
\[%
\begin{array}
[c]{rl}%
\left(  i\right) & \displaystyle (\phi, \varphi) \mapsto g\left(t,\phi ,\varphi,y, \hat{y}\right)
\text{ is continuous,}\smallskip \\
\left(  ii\right)  & \displaystyle|g(t, \phi,\varphi, y,\hat{y})-g(t,\phi,\varphi,y^{\prime},\hat{y}%
)|\leq\tilde{L}|y-y^{\prime}|,\medskip\\
\left(  iii\right)  & \displaystyle|g(t,\phi,\varphi,y,\hat{y})-g(t,\phi,\varphi, y,\hat{y}^{\prime}%
)|^{2}\leq\tilde{L}_1 ( \varphi)  \int_{-\delta}^{0}\left\vert \hat
{y}(\theta)-\hat{y}^{\prime}(\theta)\right\vert ^{2}\tilde{\rho}%
(d\theta),\medskip\\
\left(  iv\right)  & \displaystyle\left\vert g\left( t,\phi ,\varphi,0,0\right)
\right\vert <\tilde{M}(1+\left\Vert \phi \right\Vert  + \Vert \varphi \Vert)
\end{array}
\]
for any $t\in \lbrack
0,T]$, $\phi \in {\Lambda }$, $\varphi \in \Gamma$, $y,y^{\prime}\in\mathbb{R}$ and $\hat{y},\hat{y}^{\prime}\in L^{2}\left(
[-\delta,0];\mathbb{R}\right)  $.

\end{itemize}

These assumptions on coefficients are imposed throughout the paper in order to enforce the existence and uniqueness of the solution to the FBSDE \eqref{FBSDE}.

The spaces for the solutions of the FBSDE \eqref{FBSDE} are defined as follows:

\begin{itemize}
\item[\textrm{(i)}] For $m \geq 1$ and a subset $B \subset \mathbb{R}^m$, $\mathbb{S}^2 (B) $  denotes the space of continuous
$\mathbb{F}$-progressively measurable processes $X:\Omega\times\left[
0,T\right]  \rightarrow B$ such that
$$\mathbb{E} \Bigg[ \sup_{s \in [0,T]}|X(s)|^2  \Bigg] < \infty \, ; $$

\item[\textrm{(ii)}] $\mathbb{H}^2 $
denotes the space of
$\mathbb{F}$-progressively measurable processes $Z:\Omega\times\left[
0,T\right]  \rightarrow\mathbb{R}^{d'}$ such that
\[
\mathbb{E}\left[  \int_{0}^{T}|Z(s)|^{2}ds\right]<+\infty\,.
\]
\end{itemize}

\section{The Forward Reflected SDE}
\label{sec:forward}

In this section, we focus on the forward Eq. of system \eqref{FBSDE}
\begin{equation}
   \label{forward2}
\left\{
\begin{array}{l}
\displaystyle X^{t,\psi}\left( s\right) =\phi \left( t\right) - \varphi \left( t\right)
+\int_{t}^{s}b(r,X^{t,\psi})dr+\int_{t}^{s}\sigma (r,X^{t,\psi})dW\left(
r\right) + K^{t,\psi} (s) \medskip \\
\displaystyle
K^{t,\psi}(s) = \varphi
 (t) + \int_{t}^{s}\nu(X^{t,\psi}\left( r\right) ) \mathds{1}_{\left\{ X^{t,\psi
}\left( r\right) \in \partial G\right\} } dA^{t,\psi
} \left( r\right), \medskip
\end{array}%
\right.
\end{equation}%
where $K^{t, \psi}$ is a bounded variation process and $A^{t,\psi} (s) = \Vert K^{t,\psi} \Vert_{t,s}$ for $s \in [t,T]$. Recall that $\psi$ denotes the couple of initial paths $(\phi, \varphi) \in \Lambda \times \Gamma$.

Correspondingly, we will employ the notation $\mathbf{X}^{t,\psi} = ({X}^{t,\psi},K^{t,\psi})$ for the solution of the above reflected equation, especially to emphasize the dependence of the BSDE's coefficients on it.

\begin{proposition}[Existence and Uniqueness]%
\label{prop:forward}
Let assumption (A$_1$) hold and let the domain $G\subset\mathbb{R}^d$ be open, bounded and convex.
Then equation~\eqref{forward2} admits a unique strong solution
$\mathbf X^{t,\psi}\in\mathbb{S}^2(\bar G)\times\mathbb{S}^2(\mathbb{R}^d)$.
\end{proposition}

The result is a direct consequence of the classical theory of reflected SDEs in convex domains.
Indeed, owing to convexity, the pair $(G,\nu)$ satisfies the conditions of Lions--Sznitman~\cite[Thm.~3.1]{Sznitman} and Saisho~\cite[Thm.~4.1]{Saisho}:
the boundary $\partial G$ is of class $C^2_b$ (via the defining function $\ell$), the coefficients $b$ and $\sigma$ are Lipschitz continuous on $\Lambda$ (assumption~(A$_1$)), and the normal vector $\nu(x)=\nabla\ell(x)$ is Lipschitz on $\partial G$ with $|\nu(x)|=1$.
Under these assumptions the Skorokhod map is Lipschitz in the uniform norm and the standard Picard iteration yields a unique strong solution.

Proposition \ref{prop:forward} is also a particular case of Theorem 2.3 in \cite{zalinescu} and therefore the proof is omitted.

The following result represents the path-dependent analogue of Propositions 3.1 and 3.2 in \cite{pa-zh/98}.
\begin{proposition} \label{prop:generalized_3.1_corrected}
Under assumption $(A_1)$, for each $p > 1$ there exists a constant \( C_p > 0 \) such that for each \( 0 \leq t \leq T \),  for all initial paths $ \phi_1, \phi_2 \in \Lambda$ and $\varphi \in \Gamma$, we have:
\begin{equation}
    \label{est1}
\mathbb{E} \left[ \sup_{s \in [0,T]} \left(| X^{t,(\phi_1, \varphi)}(s) - X^{t,(\phi_2,  \varphi)}(s)| + | A^{t,(\phi_1,\varphi)}(s) - A^{t,(\phi_2,\varphi)}(s) |\right)^p  \right] \leq C_p \left\| \phi_1 - \phi_2 \right\|^p \, .
\end{equation}

Moreover, for any \( q > 0 \) there exists a constant $C(q)$ such that for any $\psi \in \Lambda \times \Gamma$:
\begin{equation}
\label{estimate2}
 \mathbb{E}\left[ e^{q A^{t,\psi}(T)} \right] < C (q) .
\end{equation}
\end{proposition}

\begin{proof}
The convexity of $G$ implies that
\begin{equation}
\label{inequality7}
   \langle x_2 - x_1 , \nabla\ell (x_1 ) \rangle \geq 0 \, ,\forall x_1 \in \partial G, \forall x_2 \in G.
\end{equation}

So we can proceed similarly to Prop. 3.1 of \cite{pa-zh/98} or to Lemma 3.1 of \cite{Sznitman} (p. 524-525). By applying It\^o's formula  and by Doob's inequality, the following holds for any $p>1$:
\begin{multline}
\label{inequality8}
\mathbb{E} \left[ \sup_{s \in [0,T]} | X^{t,(\phi_1, \varphi)}(s) - X^{t,(\phi_2,  \varphi)}(s)|^p   \right] \leq C_{1,p} \left\| \phi_1 - \phi_2 \right\|^{p} \\
+ C_{2,p} \mathbb{E} \int_0^s | X^{t,(\phi_1, \varphi)}(r) - X^{t,(\phi_2,  \varphi)}(r)|^p  dr  \, .
\end{multline}

From Gronwall's lemma, we conclude that
\begin{equation}
    \label{3.2forX}
    \mathbb{E} \left[ \sup_{s \in [0,T]} {| X^{t,(\phi_1, \varphi)}(s) - X^{t,(\phi_2,  \varphi)}(s)|  }^p  \right] \leq C_{3,p} \left\| \phi_1 - \phi_2 \right\|^p \, .
\end{equation}

Concerning $A^{t, \psi}$, a similar argument as in \cite{pa-zh/98} (p.551) shows that
\begin{multline}
\label{inequality9}
A^{t,(\phi,\varphi)}(s) = \ell \left( X^{t,(\phi,\varphi)}(s) \right) - \ell \left( \phi(t) \right) - \int_t^s \mathcal{L} \, \, \ell   \left( X^{t,(\phi,\varphi)}(r) \right)  dr \\
- \int_t^s \nabla \ell  \left(  X^{t,(\phi,\varphi)}(r)\right) \sigma \left( r, X^{t,(\phi,\varphi)}\right) dW (r),
\end{multline}
where $\mathcal{L}$ is the second order operator defined in Eq. \eqref{operator}.

In the same way as in Prop. 3.2 in \cite{pa-zh/98}, by Eq. \eqref{3.2forX} and Eq. \eqref{inequality9}, we deduce that
\begin{equation}
    \label{3.2forK}
    \mathbb{E} \left[ \sup_{s \in [0,T]} {| A^{t,(\phi_1, \varphi)}(s) - A^{t,(\phi_2,  \varphi)}(s)|  }^p  \right] \leq C_{4,p} \left\| \phi_1 - \phi_2 \right\|^p \, .
\end{equation}

Combining Eq. \eqref{3.2forX} and Eq. \eqref{3.2forK} proves estimate \eqref{est1}.

To prove the exponential bound \eqref{estimate2}, we use that
$$
A_s^{t,\psi}
\;\;=\;\;
\int_t^s \nabla \ell (X^{t,\psi} _r)
\,dK_r^{t,\psi}.
$$
We define the process
\[
Y_s = \exp\left( - q \int_t^s \nabla \ell \left( X^{t,\psi}(r) \right) \sigma\left( r, X^{t,\psi} \right) dW(r) \right),
\]
with \( q > 0 \). Since the Novikov's condition
$$
\mathbb{E} \left[\exp  \dfrac{1}{2} \int_t^T  q^2 | \nabla \ell (X^{t,\psi}(s))  \sigma (s, X^{t,\psi}) |^2 ds   \right] < + \infty
$$
is clearly satisfied, it follows that
$$
Y_s  \left[\exp  \left\{ - \dfrac{1}{2} \int_t^s  q^2 | \nabla \ell (X^{t,\psi}(r)) \sigma (r, X^{t,\psi}) |^2 dr \right\}   \right]
$$
is a martingale.

This implies that
$$
\mathbb{E} \left[ Y_T   \, \exp  \left\{ - \dfrac{1}{2} \int_t^T  q^2 | \nabla \ell (X^{t,\psi}(r)) \sigma (r, X^{t,\psi})dr | ^2\right\}   \right] = 1.
$$

 From Eq. \eqref{inequality9} we infer that
$$ \mathbb{E}\left[ e^{q A^{t,\psi}(T)} \right] = \mathbb{E} \left[ Y_T  \, \exp \,  q \left\{
\ell \left( X^{t,(\phi,\varphi)}(T) \right) - \ell \left( \phi(t) \right) - \int_t^T \mathcal{L} \, \ell   \left( X^{t,(\phi,\varphi)}(r) \right)  dr  \right\} \right]  . $$

Since $\ell ({X^{t,\psi}}) $ is bounded, estimate \eqref{estimate2} is proved.
\end{proof}

\subsection{Tower Property}

 We enlarge the state space and look at $\mathbf{X}^{t,\psi}$ defined in Eq. \eqref{forward2} as a process of the path (i.e. an infinite-dimensional state space) to recover the {\it Markov} property for path-dependent SDE.

\begin{proposition}[Tower Property]
\label{tower_prop}
Given an initial data $t \in [0,T]$ and $\psi \in \Lambda \times \Gamma$, the following equality holds for $t \leq r \leq T$, a.s.:
\begin{equation}
    \label{tower1}
\mathbf{X}^{t,\psi} =   \mathbf{X}^{r, \mathbf{X}^{t,\psi}} \, .
\end{equation}
\end{proposition}

This type of result is quite standard for It\^{o} diffusions. For the sake of the reader, we provide a proof for the above context.

\begin{proof}
Let us fix $t\in\lbrack0,T]$ and $r\in\lbrack t,T]$. We have, for $\psi
=(\phi,\varphi)\in\Lambda\times\Gamma$, $\mathbb{P}$-a.s.,
\[
\left\{
\begin{array}
[c]{l}%
\displaystyle X^{r,\psi}(s)=\phi(r)-\varphi(r)+\int_{r}^{s}b\left(  u,X^{r,\psi}(u)\right)  du+\int%
_{r}^{s}\sigma\left(  u,X^{r,\psi}(u)\right)  dW(u)+K^{r,\psi}(s),\quad\\
\hfill\forall s\in\lbrack r,T];\medskip\\
\displaystyle K^{r,\psi}(s)=\varphi(r)+\int_{r}^{s}\nu(X^{r,\psi}(u))\mathds{1}_{\{X^{r,\psi
}(u)\in\partial G\}}dA^{r,\psi}(s),\ \forall s\in\lbrack r,T];\medskip\\
\mathbf{X}^{r,\psi}(s)=(\phi(s),\varphi(s)),\ \forall s\in\lbrack0,r].\medskip
\end{array}
\right.
\]
We would like to show that we can replace $\psi$ by $\mathbf{X}^{t,\psi}$.
This is straightforward for the deterministic integrals, the only difficulty
occuring from the stochastic integral, which is not pathwise defined.

In this regard, we first remark that the continuous process $\mathbf{X}^{r,\mathbf{X}^{t,\psi
}}$ is $\mathbb{F}^{0}$-adapted, so the stochastic integral $\int_{r}%
^{s}\sigma\bigl(  u,X^{r,\mathbf{X}^{t,\psi}}(u)\bigr)  dW(u)$ is well-defined. Indeed, for fixed $s\in[r,T]$ and $\bar{\psi}\in\Lambda\times\Gamma$, we have $X^{r,\mathbf{X}^{t,\bar{\psi}}}=X^{r,\mathbf{X}^{t,\bar{\psi}}(\cdot\wedge r)}$ and
the mappings $(\omega,\psi)\mapsto \mathbf{X}^{r,\psi}(\omega,s)$ and $\mathbf{X}^{t,\bar{\psi}}(\cdot\wedge r)$ are $\mathcal{F}^r_s\otimes \mathcal{B}(\Lambda\times\Gamma)/\mathcal{B}(\mathbb{R}^d)$-measurable, respectively $\mathcal{F}^t_r/\mathcal{B}(\Lambda\times\Gamma)$-measurable.

Let us denote $I^{\psi}(s):=\int_{r}^{s}\sigma\left(
u,X^{r,\psi}(u)\right)  dW(u)$, $s\in\lbrack r,T]$, for $\psi=(\phi
,\varphi)\in\Lambda\times\Gamma$. For every $N\in\mathds{N}$, let us consider
the partition $r=t_{0}<t_{1}<\ldots<t_{N}=s$. We approximate $I^{\psi}$ by
\[
I^{N,\psi}(s):=\sum_{k=1}^{N}\sigma\left(  t_{k-1},X^{r,\psi}(t_{k-1})\right)
\left(  W(t_{k})-W(t_{k-1})\right)  .
\]

Obviously, $I^{N,\psi}(s)\rightarrow I^{\psi
}(s)$ as $N\rightarrow\infty$ in probability for every $\psi\in\Lambda
\times\Gamma$.

Let us now fix $\bar{\psi}\in\Lambda\times\Gamma$. For every $\varepsilon>0$,
we have that
\begin{align*}
\mathbb{P}(\lvert I^{N,\mathbf{X}^{t,\bar{\psi}}}(s)-I^{\mathbf{X}%
^{t,\bar{\psi}}}(s)\rvert\geq\varepsilon)  & =\mathbb{E}[\mathds{1}_{\{\lvert
I^{N,\mathbf{X}^{t,\bar{\psi}}}(s)-I^{\mathbf{X}^{t,\bar{\psi}}}(s)\rvert
\geq\varepsilon\}}]\\
& =\mathbb{E}\left[  \mathbb{E}[\mathds{1}_{\{\lvert I^{N,\mathbf{X}%
^{t,\bar{\psi}}(\cdot\wedge r)}(s)-I^{\mathbf{X}^{t,\bar{\psi}}(\cdot\wedge
r)}(s)\rvert\geq\varepsilon\}}\mid\mathcal{F}_{r}^{t}]\right]  .
\end{align*}
since $X^{r,\mathbf{X}^{t,\bar{\psi}}}=X^{r,\mathbf{X}^{t,\bar{\psi}}%
(\cdot\wedge r)}$. But $I^{N,\psi}(s)-I^{\psi}(s)$ is $\mathcal{F}_{s}^{r}%
$-measurable for every $\psi\in\Lambda\times\Gamma$ and $\mathbf{X}%
^{t,\bar{\psi}}(\cdot\wedge r)$ is $\mathcal{F}_{r}^{t}$-measurable.

By the
independence of $\mathcal{F}_{s}^{r}$ and $\mathcal{F}_{r}^{t}$, we get
\begin{align*}
\mathbb{E}[\mathds{E}_{\{\lvert I^{N,\mathbf{X}^{t,\bar{\psi}}(\cdot\wedge
r)}(s)-I^{\mathbf{X}^{t,\bar{\psi}}(\cdot\wedge r)}(s)\rvert\geq\varepsilon
\}}\mid\mathcal{F}_{r}^{t}]  & =\mathbb{E}[\mathds{E}_{\{\lvert I^{N,\psi
}(s)-I^{\psi}(s)\rvert\geq\varepsilon\}}]_{\mid\psi=\mathbf{X}^{t,\bar{\psi}%
}(\cdot\wedge r)}\\
& =\mathbb{P}(\lvert I^{N,\psi}(s)-I^{\psi}(s)\rvert\geq\varepsilon)_{\mid
\psi=\mathbf{X}^{t,\bar{\psi}}(\cdot\wedge r)},
\end{align*}
which converges to $0$ as $N\rightarrow\infty$. Therefore, by the Lebesgue's
dominated convergence theorem,
\[
\lim_{N\rightarrow\infty}\mathbb{P}(\lvert I^{N,\mathbf{X}^{t,\bar{\psi}}%
}(s)-I^{\mathbf{X}^{t,\bar{\psi}}}(s)\rvert\geq\varepsilon)=0,\ \forall
\varepsilon>0.
\]
But $I^{N,\mathbf{X}^{t,\psi}}(s)$ converges in probability to $\int_{r}%
^{s}\sigma\bigl(  u,X^{r,\mathbf{X}^{t,\psi}}(u)\bigr)  dW(u)$ as
$N\rightarrow\infty$, so $\int_{r}^{s}\sigma\bigl(  u,X^{r,\mathbf{X}^{t,\psi
}}(u)\bigr)  dW(u)=I^{\mathbf{X}^{t,\bar{\psi}}}(s)$, $\mathbb{P}$-a.s.

We have proven that we can replace $\psi$ by $\mathbf{X}^{t,\psi}$, so we get
\[
\left\{
\begin{array}
[c]{l}%
\displaystyle X^{r,\mathbf{X}^{t,\psi}}(s)=X^{t,\psi}(r)+\int_{r}^{s}b\left(
u,X^{r,\mathbf{X}^{t,\psi}}(u)\right)  du+\int_{r}^{s}\sigma\left(
u,X^{r,\mathbf{X}^{t,\psi}}(u)\right)  dW(u)\quad\\
\hfill+K^{r,\mathbf{X}^{t,\psi}%
}(s),\ \forall s\in\lbrack r,T];\medskip\\
\displaystyle K^{r,\mathbf{X}^{t,\psi}}(s)=K^{t,\psi}(r)+\int_{r}^{s}\nu(X^{r,\mathbf{X}%
^{t,\psi}}(u))\mathds{1}_{\{X^{r,\mathbf{X}^{t,\psi}}(u)\in\partial
G\}}dA^{r,\mathbf{X}^{t,\psi}}(s),\ \forall s\in\lbrack r,T];
\end{array}
\right.
\]
By uniqueness, allowing now $s$ to vary, we can ensure that the following equalities hold for every $s\in[r,T]$, $\mathbb{P}(d\omega)$-a.s.:
\begin{align}
X^{r, \mathbf{X}^{t,\psi} ( \omega)}(s,\omega)  & =X^{t,\psi}(s,\omega),\nonumber\\
K^{r, \mathbf{X}^{t,\psi} ( \omega)}(s,\omega)  & =K^{t,\psi}(s,\omega) \, .\nonumber
\end{align}
The above equalities clearly hold for $s\in[0,r)$, so we can conclude the proof.
\end{proof}

Tower property \eqref{tower1} represents the first passage to state the Markov property. The other necessary element is represented by the definition of a proper transition semigroup that can handle paths.

\subsection{Markov property and Transition Semigroup}

The reflection term
$K^{t,\psi}$ introduces a dependence on the path of the process at the boundary, where the reflection occurs. The transition semigroup for a reflected SDE \eqref{forward2} can be extended to the coupled process
$ \mathbf{X}^{t,\psi} = (X^{t,\psi}, K^{t,\psi})$
  where the semigroup models the joint evolution of both processes.

We introduce the transition semigroup $P_{t,s}$ for the pair $(X,K)$
acting on the product space $\Lambda \times \Gamma$.
We define $P_{t,s}:\mathcal{B}_{b}\left(  \Lambda  \times  \Gamma \right)
\rightarrow \mathcal{B}_{b}\left(  \Lambda  \times  \Gamma \right) $, where $\mathcal{B}_b (\Lambda \times \Gamma)$ is the space of Borel bounded functions on $ \Lambda \times \Gamma$, by
\begin{equation}
\label{PtG}
P_{t,s}\left[  f\right]  (\psi):=\mathbb{E}\left[  f\left(  \mathbf{X}^{t,\psi
}(\cdot\wedge s)\right)  \right]
\,,\qquad f\in\mathcal{B}_{b}\left(  \Lambda
\times \Gamma \right) \, .
\end{equation}

\noindent
The transition probabilities are given by
\begin{equation}
p(t,\psi,s,B) :=P_{t,s}\left[  \mathds{1}_{B}\right]  (\psi)  =\mathbb{P}\left(  \,
\mathbf{X}^{t,\psi}(\cdot\wedge s)\in B\right)
 \, ,
\end{equation}
for $0 \leq t \leq s \leq T$, $\psi\in \Lambda \times \Gamma$ and for all Borel sets $B$ in $  \Lambda \times \Gamma $.

\begin{proposition}[Markov Property] Let $\mathbf{X}^{t,\psi}$ be a strong solution of Eq.
\eqref{forward2}. Then $\mathbf{X}^{t,\psi}$ is a Markov process on
$\Lambda \times \Gamma$, i.e.
\begin{equation}
\label{markov1}
\mathbb{P}\left(  \mathbf{X}^{t,\psi}\in B\,|\,\mathcal{F}_{s}^{t}\right)
=\mathbb{P}\left(  \mathbf{X}^{t,\psi}\in B\,|\,\mathbf{X}^{t,\psi}(\cdot\wedge
s)\right)\, \text{a.s.,}
\end{equation}
for all $0\leq t\leq s\leq T$, for all $\psi\in \Lambda \times \Gamma$ and for all Borel sets $B$ in $  \Lambda \times \Gamma $.

\noindent
Moreover, for any $(t,r,s,\psi) \in[0,T]^3\times (\Lambda \times \Gamma) $ with $t\leq r\leq s$
 and any Borel set $B$ in $\Lambda \times \Gamma$, the following holds a.s. on $\Omega$:
\begin{align}
\mathbb{P}\left(   \mathbf{X}^{t,\psi}(\cdot\wedge s)  \in B|\mathcal{F}^t_{r}\right)    & =\mathbb{P}\left(  \,  \mathbf{X}^{t,\psi}(\cdot\wedge s)  \in B \, | \, \mathbf{X}^{t,\psi}(\cdot\wedge r)\right)  \nonumber\\
& =p\left(  r,
\mathbf{X}^{t,\psi}(\cdot\wedge r)
,s,B\right).
\end{align}
\end{proposition}

\begin{proof}
Let $\psi\in \Lambda \times \Gamma$ and $t\leq r\leq s$. By definition \ref{PtG}, we have
\begin{align}
P_{r,s}\left[  f\right]  (\mathbf{X}^{t,\psi}(\omega^{\prime}))&=\mathbb{E}_{\omega}\left[  f\left(  \mathbf{X}^{r,\mathbf{X}^{t,\psi
}(\omega^{\prime})}(\omega,\cdot\wedge s)\right)  \right]
=\mathbb{E}_{\omega}\left[  f\left(  \mathbf{X}^{r,\mathbf{X}^{t,\psi
}(\omega^{\prime},\cdot\wedge r)}(\omega,\cdot\wedge s)\right)  \right]\nonumber\\
& =\mathbb{E}_{\omega}\left[  k(\omega,\mathbf{X}^{t,\psi}(\omega^{\prime},\cdot\wedge r))\right],  \,
\forall \omega^{\prime}\in\Omega,\nonumber
\end{align}
where $k:\Omega\times\Lambda\times\Gamma\mapsto\mathbb{R}$ denotes the process defined by
\[
k(\omega,\phi,\varphi):=f(\mathbf{X}^{r,(\phi,\varphi)}(\omega,\cdot\wedge s)).
\]
Since $k(\cdot,\mathbf{X}^{t,\psi
}(\omega^{\prime},\cdot\wedge r))$ is
measurable with respect to the $\sigma$-algebra $\mathcal{F}_{s}^{r}$ for all $\omega^{\prime}\in\Omega$, and $\mathcal{F}_{s}^{r}$ is
independent from $\mathcal{F}_{r}^{t}$, we have
\[
\mathbb{E}_{\omega}\left[  k(\omega
,\mathbf{X}^{t,\psi
}(\omega^{\prime},\cdot\wedge r)\right]=\mathbb{E}\left[  k(\cdot,\mathbf{X}^{t,\psi
}(\cdot,\cdot\wedge r))\mid\mathcal{F}_{r}^{t}\right](\omega^{\prime}),\ \mathds{P}(d\omega^{\prime})\text{-a.s.}
\]
But $k(\cdot,\mathbf{X}^{t,\psi
}(\cdot,\cdot\wedge r))  =k(\cdot,\mathbf{X}^{t,\psi
}) $ by the tower property \eqref{tower1} for the forward process. We can now apply this to get
\[
P_{r,s}\left[  f\right]  (\mathbf{X}^{t,\psi})=\mathbb{E}\left[  f\left(
\mathbf{X}^{t,\psi}(\cdot\wedge s)\right) |\mathcal{F}%
_{r}^{t}\right] \, .
\]

If $f$ is the indicator function $\mathds{1}_{B}$, for a Borel set $B$ in $  \Lambda \times \Gamma $, we obtain
\begin{align}
p\left(  r,
\mathbf{X}^{t,\psi}(\cdot\wedge r)
,s,B\right)&=P_{r,s}\left[  \mathds{1}_{B}\right]  (\mathbf{X}^{t,\psi}(\cdot\wedge r))=P_{r,s}\left[  \mathds{1}_{B}\right]  (\mathbf{X}^{t,\psi})\nonumber\\
&=\mathbb{E}%
\left[  \mathds{1}_{B}\left(  \mathbf{X}^{t,\psi}(\cdot\wedge s)\right)  |\mathcal{F}_{r}^{t}\right]
=\mathbb{P}\left(  \mathbf{X}^{t,\psi}
(\cdot\wedge s)\in B|\mathcal{F}_{r}^{t}\right).\nonumber
\end{align}

\noindent
Conditioning the above to $\mathbf{X}^{t,\psi}(\cdot\wedge r)$, we get
\[
p\left(  r,
\mathbf{X}^{t,\psi}(\cdot\wedge r)
,s,B\right)=\mathbb{P}\left(  \mathbf{X}^{t,\psi}(\cdot\wedge s)\in
B\,|\,\mathcal{F}_{r}^{t}\right)  =\mathbb{P}\left(  \mathbf{X}^{t,\psi}(\cdot\wedge
s)\in B\,|\,\mathbf{X}^{t,\psi}(\cdot\wedge r)\right),
 \]
which concludes the second part of the statement.
For the first part (the Markov property), one can consider in particular $s=T$.
\end{proof}

\section{The Forward Backward system}
\label{sec:fbsde}

In this section, we focus on the BSDE of the FBSDE system \eqref{FBSDE}.
\begin{equation}
\left\{
\begin{array}{l}
\displaystyle Y^{t,\psi}\left( s\right) =h({X}^{t,\psi}, {K}^{t,\psi})+\int_{s}^{T}f(r,{X}^{t,\psi},{K}^{t,\psi},Y^{t,\psi}\left( r\right) ,Z^{t,\psi}\left(
r\right) ,Y_{r}^{t,\psi})dr \smallskip  \\
\qquad  \qquad \displaystyle +\int_{s}^{T}g(r,{X}^{t,\psi},{K}^{t,\psi},Y^{t,\psi
}\left( r\right) ,Y_{r}^{t,\psi})dA^{t,\psi}(r)-\int_{s}^{T}Z^{t,\psi}\left( r\right) dW\left( r\right) , \medskip  \\
\displaystyle Y^{t,\psi}\left( s\right) =Y^{s,\psi }\left( s\right) .\quad
\end{array}%
\right.   \label{bsde9}
\end{equation}%

We apply Theorem 4 from \cite{spa} that ensures there exists a unique solution in $\mathbb{S}^{2} (\mathbb{R})\times\mathbb{H}^{2}$  for the BSDE  \eqref{bsde9} by exploiting a standard Banach's fixed point argument which requires Lipschitz constants
$L$ or $\delta$ to be small enough.

For $\varphi \in \Gamma$, we define the modulus of continuity as
\[
\mathbf{\omega}_{\delta} (\varphi):=\sup_{s\in\lbrack0,T-\delta]}
\Vert
\varphi
\Vert_{s,s+\delta}  \, ,
\]
with $L_1 (\varphi)$ and $\tilde{L}_1 (\varphi)$ being the stochastic processes modelling the Lipschitz constants introduced in assumptions $\mathrm{(A}_{3}\mathrm{)}$ and $\mathrm{(A}_{4}\mathrm{)}$. Let $\beta > 2\sqrt{2}\tilde{L}$, we assume the existence of a positive constant $c<c_{\beta,\tilde{L}%
}:=\min\left\{  \frac{\beta^{2}-8\tilde{L}^{2}}{4\beta^{2}},\frac{1}%
{584}\right\}  \ $such that

\begin{description}
\item[$\mathrm{(H}_{1}\mathrm{)}$]
$\displaystyle  L_{1} (\varphi) \cdot\max\left\{  1,T\right\}
\cdot\frac{e^{(8L^{2}+\frac{1}{2})\delta+\beta\mathbf{\omega}_{\delta} (\varphi)}%
}{4L^{2}}\leq c,\quad\mathbb{P}$-a.s.;

\item[$\mathrm{(H}_{2}\mathrm{)}$] $ \displaystyle 4\tilde{L}_{1} (\varphi) \cdot \Vert \varphi \Vert_{0,T} \cdot
\frac{e^{(8{L}^{2}+\frac{1}{2})\delta+\beta\omega_\delta (\varphi)}}{\beta}\leq
c,\quad\mathbb{P}$-a.s.
\end{description}

\begin{proposition}
\label{continuous_dep}
Let us consider the coupled forward--backward system \eqref{FBSDE}
which satisfies Assumptions $(A_1)$-$(A_4)$ and $(H_1)$-$(H_2)$.
Then, the coupled forward--backward system admits a unique solution
$\left(X^{t,\psi},K^{t,\psi}, Y^{t,\psi}, Z^{t,\psi} \right)$ in $\mathbb{S}^2 (\bar{G}) \times \mathbb{S}^2 (\mathbb{R}^d)   \times \mathbb{S}^{2} (\mathbb{R})\times\mathbb{H}^{2} $. Furthermore, the map
$(t, \psi) \mapsto \left(X^{t,\psi},K^{t,\psi}, Y^{t,\psi}, Z^{t,\psi}\right)$
is continuous.
\end{proposition}

\begin{proof}
The existence and uniqueness for the forward part is given in Prop. \ref{prop:forward}. For the BSDE, we note that $h (\mathbf{X}^{t,\psi})$ is an adapted random variable of the whole forward process with values in $L^{2}\left(  \Omega,\mathcal{F}_{T},\mathbb{P};\mathbb{R} \right)$. Thus, we can treat it as a terminal condition.

We set $h (\mathbf{X}^{t,\psi}) = \xi$ and we consider the map $\Psi:\mathbb{S}^{2} (\mathbb{R})\times\mathbb{H}^{2}\rightarrow\mathbb{S}^{2} (\mathbb{R})\times\mathbb{H}^{2}$, defined in the following way: for $\left(
U,V\right)  \in\mathbb{S}^{2} (\mathbb{R})\times\mathbb{H}^{2}$, $ \Psi\left(  U,V\right)  =\left(  Y,Z\right)  $, where the couple of
adapted processes $\left(  Y,Z\right)  $ is the solution to the equation
\begin{multline}
\label{BSDE iterative 1}
Y(t)=\xi+\int_{t}^{T}f(s,Y(s),Z(s),U_{s},V_{s})ds \\
+\int_{t}^{T}g(s,U(s),U_{s}%
)dA(s)-
\int_{t}^{T}Z(s)dW(s),\quad t\in\left[  0,T\right]  .
\end{multline}
 We proceed as in \cite{spa} where existence and uniqueness are obtained by the Banach fixed point
theorem.

Concerning the continuity of the solution map, we establish that the following expression is uniformly bounded with respect to $\psi \in \Lambda \times \Gamma$
\begin{multline}
\label{tobound}
\mathbb{E}\left[\left(\int_t^T e^{\beta A^{t,\psi}(s)} \left|f(s, \mathbf{X}^{t,\psi}, 0, 0, 0)\right|^2 ds\right)^p \right. \\
\left. + \left(\int_t^T e^{\beta A^{t,\psi}(s)} \left|g(s, \mathbf{X}^{t,\psi}, 0, 0)\right|^2 dA^{t,\psi}(s)\right)^p\right] \, .
\end{multline}

\noindent
For the first term, we apply H\"older's inequality:
\begin{multline}
\mathbb{E}\left[\left(\int_t^T e^{\beta A^{t,\psi}(s)} |f(s, \mathbf{X}^{t,\psi}, 0, 0, 0)|^2 ds\right)^p\right] \\
\leq (T-t)^{p-1}\mathbb{E}\left[\int_t^T e^{p\beta A^{t,\psi}(s)} |f(s, \mathbf{X}^{t,\psi}, 0, 0, 0)|^{2p} ds\right] \, .
\end{multline}
By $(A_3)(iv)$, we have
\begin{align}
\mathbb{E}  \Bigg[\int_t^T e^{p\beta A^{t,\psi}(s)} & |f(s, \mathbf{X}^{t,\psi}, 0, 0, 0)|^{2p} ds\Bigg] \leq M^{2p}\mathbb{E}\left[\int_t^T e^{p\beta A^{t,\psi}(s)}(1 + |\mathbf{X}^{t,\psi}|)^{2p} ds\right] \\
& \leq M^{2p}\mathbb{E}\left[\sup_{t\leq s\leq T}(1 + |\mathbf{X}^{t,\psi}(s)|)^{4p}\right]^{1/2} \cdot \mathbb{E}\left[\int_t^T e^{2p\beta A^{t,\psi}(s)}ds\right]^{1/2}\\
&\leq M^{2p}\mathbb{E}\left[\sup_{t\leq s\leq T}(1 + |\mathbf{X}^{t,\psi}(s)|)^{4p}\right]^{1/2} \cdot \sqrt{C(2p\beta)(T-t)}
\end{align}

\noindent
From the moment estimates given in Prop. \ref{prop:generalized_3.1_corrected}, we have
\begin{equation}
\mathbb{E}\left[\left(\int_t^T e^{\beta A^{t,\psi}(s)} \left|f(s, \mathbf{X}^{t,\psi}, 0, 0, 0)\right|^2 ds\right)^p\right] \leq C\Vert \psi \Vert^{2p}
\label{f_ineq}
\end{equation}

\noindent
For the second term, applying $(A_4)(iv)$
\begin{align}
\int_t^T & e^{\beta A^{t,\psi}(s)}  \left|g(s, \mathbf{X}^{t,\psi}, 0, 0)\right|^2 dA^{t,\psi}(s) \leq\\ & \leq M^{2} \left( \sup_{t\leq s\leq T} |g(s, \mathbf{X}^{t,\psi}, 0, 0) |^4 A^{t,\psi}(T) \right)^{1/2} \cdot \left(\int_t^T e^{2\beta A^{t,\psi}(t)}dA^{t,\psi}(s)\right)^{1/2}\\
&\leq \dfrac{M^{2}}{\beta}\left(\sup_{t\leq s\leq T}(1 + |\mathbf{X}^{t,\psi}(s)|)^{4} A^{t,\psi} (T)\right)^{1/2} \cdot  e^{2\beta A^{t,\psi}(T)}
\label{inequa}
\end{align}

Taking expectation and power $p$ of \eqref{inequa} and combining with \eqref{f_ineq}, we arrive at bounding Eq. \eqref{tobound} by $C \Vert \psi \Vert^{2p}$. Similarly, by $(A_2)$, we have
\begin{equation}
\label{estimate1}
\mathbb{E}\left[ e^{p\beta A^{t,\psi}(T)}\left| h(\mathbf{X}^{t,\psi}(T) ) \right|^{2p} \right] \leq C\Vert \psi \Vert^{2p} \, .
\end{equation}

Hence, the uniform bounds, with respect to $\psi$, allow us to apply Theorem 4 from \cite{spa}. The continuity of the mapping $(t,\psi) \mapsto (Y^{t,\psi}, Z^{t,\psi})$ then follows from this theorem.
\end{proof}

\subsection{The Feynman-Kac formula}

In this section, we provide a representation for the solution of the  BSDE \eqref{bsde9}.

By defining
\begin{equation}
u(t,\psi ):=Y^{t,\psi}(t),
\end{equation}
our first goal is to prove the following relation:
\begin{equation}
Y^{t,\psi}(s)=u(s,\mathbf{X}^{t,\psi}),\ \forall t,s\in \lbrack 0,T],
\end{equation}
known as the Feynman-Kac formula.

An intermediate step to prove the above formula
is related to the tower property for the process $Y^{t,\psi}$, which we directly prove via the composition procedure similar to Prop. \ref{tower_prop}.

\begin{proposition}[Tower property for $Y$ and $Z$]
Given the BSDE \eqref{bsde9}, the following relations
\begin{equation}
    \label{towerY}
Y^{t,\psi} (s) = Y^{r, \mathbf{X}^{t,\psi} } (s)
\end{equation}
and
\begin{equation}
    \label{towerZ}
Z^{t,\psi} (s) =Z^{r, \mathbf{X}^{t,\psi} } (s) \, , \qquad \text{for a.e.} \, s \in [r,T]
\end{equation}
hold with $t \leq r \leq s \leq T$.
\end{proposition}

\begin{proof}
We consider the BSDE starting at time $r \in (t,T]$, where the initial process $\psi$ is replaced by $(X^{r,\psi},  K^{r,\psi})$
\begin{equation}
\left\{
\begin{array}{l}
\displaystyle Y^{r,\mathbf{X}^{t,\psi} }\left( s\right) =h(X^{r,\mathbf{X}^{t,\psi}})+\int_{s}^{T}f(u,X^{r,\mathbf{X}^{t,\psi}},Y^{r,\mathbf{X}^{t,\psi}}\left( u\right) ,Z^{r,\mathbf{X}^{t,\psi}}\left(
u\right) ,Y_{u}^{r,\mathbf{X}^{t,\psi}})du \smallskip  \\
\quad  \displaystyle +\int_{s}^{T}g(u,X^{r,\mathbf{X}^{t,\psi}},Y^{r,\mathbf{X}^{t,\psi}}\left( u\right) ,Y_{u}^{r,\mathbf{X}^{t,\psi}})dA^{r,\mathbf{X}^{t,\psi}}(u)-\int_{s}^{T}Z^{r,\mathbf{X}^{t,\psi}}\left( u\right) dW\left( u\right) , \medskip  \\
\displaystyle Y^{r,\mathbf{X}^{t,\psi}}\left( s\right) =Y^{s,\mathbf{X}^{t,\psi}}\left( s\right) ,\quad
\end{array}%
\right.   \label{bsde6}
\end{equation}%

It is possible to prove \eqref{towerY} by following the same scheme we use to prove it for $X$ in Prop. \ref{tower_prop}.

{\em Step 1.} We write Eq. (5.2) with $\omega$ and $\omega^\prime$

\begin{equation}
\left\{
\begin{array}{l}
\displaystyle Y^{r,X^{t,\psi} (\omega)}\left( s , \omega^\prime\right) =h(X^{r, \mathbf{X}^{t,\psi} ( \omega)}) \smallskip  \\
\qquad  \qquad \displaystyle
 +\int_{s}^{T}f(u,X^{r, \mathbf{X}^{t,\psi} ( \omega)},Y^{r, \mathbf{X}^{t,\psi} ( \omega)}\left( u, \omega^\prime\right) ,Z^{r, \mathbf{X}^{t,\psi} ( \omega)}\left(
u, \omega^\prime\right) ,Y_{u}^{r, \mathbf{X}^{t,\psi} ( \omega)})du \smallskip  \\
\qquad  \qquad \displaystyle +\int_{s}^{T}g(u,X^{r, \mathbf{X}^{t,\psi} ( \omega)},Y^{r, \mathbf{X}^{t,\psi} ( \omega)}\left( u, \omega^\prime\right) ,Y_{u}^{r, \mathbf{X}^{t,\psi} ( \omega)})dA^{r, \mathbf{X}^{t,\psi} ( \omega)}(u) \smallskip  \\
\qquad  \qquad \displaystyle  -\int_{s}^{T}Z^{r, \mathbf{X}^{t,\psi} ( \omega)}\left( u, \omega^\prime\right) dW\left( u, \omega^\prime\right) , \medskip  \\
\displaystyle Y^{r,\mathbf{X}^{t,\psi}}\left( s\right) =Y^{s,X^{t,\psi}}\left( s\right) ,\quad
\end{array}%
\right.   \label{bsde7}
\end{equation}%
with $t \leq r \leq s \leq T$.

{\em Step 2.} The stochastic integral is analogous to the one in the forward case, the only difference being that
$Z$ is not continuous. Despite this additional issue, we still have convergence in probability and Eq. \eqref{bsde7} can be approximated following the same procedure as in Prop. \ref{tower_prop}.

{\em Step 3.} From Prop. \ref{continuous_dep} we know the solutions of BSDEs \eqref{bsde9} and \eqref{bsde7} are unique, hence they coincide.
\end{proof}

We are now in the position to state the main result of this section given by the following theorem.

\begin{theorem}[Feynman-Kac formula]
\label{th5.1}
Under hypotheses ($A_1$) - ($A_{4}$) and conditions $\mathrm{(H}_{1}\mathrm{)}$, $\mathrm{(H}_{2}\mathrm{)}$, let $(\mathbf{X}^{t,\psi}, Y^{t,\psi}, Z^{t,\psi})_{(t, \psi)\in[0,T]\times\Lambda\times \Gamma}$ solve the forward-backward system
\eqref{FBSDE}.

Let $u: [0,T] \times \Lambda \times \Gamma \rightarrow \mathbb{R}$ be the (deterministic) function defined by
\begin{equation}
  \label{FK}
 u (t, \psi) = Y^{t,\psi} (t), \quad (t, \psi) \in [0,T] \times \Lambda\times\Gamma \, .
\end{equation}
Then, $u$ is a non-anticipative continuous function and the following formula holds:
\begin{equation}
\label{kac}
Y^{t,\psi} (s)= u (s, \mathbf{X}^{t,\psi}), \ \forall s\in [0,T] ,
\end{equation}
for any $(t,\psi) \in [0,T] \times \Lambda \times \Gamma$.
\end{theorem}

\begin{proof}
Feynman-Kac formula \eqref{kac} is a direct consequence of tower property. By plugging $r=s$ in Eq. \eqref{towerY} we obtain
\begin{equation}
    \label{FK1}
Y^{t,\psi} (s) = Y^{s, \mathbf{X}^{t,\psi}  } (s) = u (s, \mathbf{X}^{t,\psi}) \, .
\end{equation}

The continuity comes from Theorem 7 in \cite{spa}.
\end{proof}

\section{The generalized directional gradient}
\label{sec:gdg}

Having established the Feynman-Kac representation for $Y^{t,\psi}$ in the previous section, we now turn to the characterization of the process $Z^{t,\psi}$.

We would like to express also the process $Z$ in terms of $X^{t,\psi}$ using a deterministic function, in the same spirit as the Feynman-Kac formula for $Y^{t,\psi}$. Intuitively, the idea concerns constructing a relation between $Z^{t,\psi}$ and a derivative with respect to the spatial variable of the function $u$ in Eq. \eqref{kac}.
In the current setting, $u$ may not be differentiable, but suitable Lipschitz assumptions on the coefficients $h$, $f$ and $g$ help to build a proper notion of gradient suitable for our setting. The generalized directional gradient introduced below provides exactly this framework for non-differentiable functionals.

\begin{definition}
\label{def:gdg}
Let $v: [0,T] \times \Lambda \times \Gamma \rightarrow \mathbb{R} $ be a continuous function. The generalized directional gradient $\nabla^{\sigma} v$ is defined as the set of all non-anticipative functions $\zeta:[0,T]\times\Lambda \times \Gamma \rightarrow\mathbb{R}^d$ such that the processes $v(\cdot,\mathbf{X}^{t,\psi})$ and $W(\cdot)$ admit a joint quadratic variation on the interval $[t, \tau]$ and the following relation
\begin{equation}
\label{gdg}
\langle v(\cdot,\mathbf{X}^{t,\psi}),W(\cdot)\rangle_{[t,\tau]}
= \displaystyle \int_t^\tau \zeta(s,\mathbf{X}^{t,\psi})ds,
\end{equation}
holds true for all $(t,\phi)\in[0,T]\times\Lambda$ and all $\tau\in[t,T]$.
\end{definition}

From \cite{fuhrman_tessitore}, we recall the definition of the {\it joint quadratic variation}  $\langle X,Y\rangle_{[t,\tau]}$  of a pair of real stochastic processes $(X(t),Y(t))_{t\in[0,T]}$ on the interval $[t,\tau]$,
\begin{equation}
\label{jqv}
\langle X,Y\rangle_{[t,\tau]}:=\lim_{\epsilon\searrow 0}\displaystyle \frac{1}{\epsilon}
\int_t^\tau (X(s+\epsilon)-X(s))\cdot(Y(s+\epsilon)-Y(s))ds,
\end{equation}
where the limit is taken in the sense of convergence in probability.

\medskip

\begin{proposition}
\label{Zprop}
If Assumptions $(A_1)-(A_4)$ and $(H_1)-(H_2)$ hold true,
the set of generalized gradients $\nabla^\sigma  u$
is not empty and for every  $\zeta \in \nabla^\sigma u$,
the following:
\begin{equation}
\label{towerZ3}
Z^{t,\psi} (s) = \zeta \left(s, \mathbf{X}^{t,\psi}\right) \, ,
\end{equation}
holds $\mathbb{P}\textit{-a.s.}$  and for
 a.e. $s \in [t,T]$.
\end{proposition}

\begin{proof}
The first equality in \eqref{towerZ3} is easily proved by plugging $s=r$ in \eqref{towerZ}.

Concerning the computation of the deterministic function $\zeta$ we proceed as Theorem 3.1 in \cite{fuhrman_tessitore}. First, we introduce a truncation operator $T_N (r) = (r \wedge N) \vee (-N)$, for $r \in \mathbb{R}$,  $N > 0$, and we define
\begin{equation}
\label{eq:xi}
\displaystyle
\zeta^N (t, \psi) = \liminf_{n \rightarrow \infty} \mathbb{E}  \left[ T_N  \left( n \int_t^{t + \frac{1}{n}} Z^{t,\psi} (s) ds \right)
\right] \, .
\end{equation}

From the independence of $\sigma$-algebras, we can pass to the conditional expectation with respect to $\mathcal{F}_s$:
\begin{multline}
\zeta^N (s,X^{t,\psi}) = \liminf_{n \rightarrow \infty} \mathbb{E}^{\mathcal{F}_s}  \left[ T_N  \left( n \int_s^{s + \frac{1}{n}} Z^{s,X^{t,\psi}} (r) dr \right)
\right] \\
= \liminf_{n \rightarrow \infty} \mathbb{E}^{\mathcal{F}_s}  \left[ T_N  \left( n \int_s^{s + \frac{1}{n}} Z^{t,\psi} (r) dr \right)
\right] \, ,
\end{multline}
where we use the tower property \eqref{towerZ} in the second equality.

We fix $t$ and $\phi$. We note that $\mathbb{P}$-a.s. by the Lebesgue differentiation theorem, the following
$$
\lim_{n \rightarrow \infty}   \left[ T_N  \left( n \int_s^{s + \frac{1}{n}} Z^{t,\psi} (r) dr \right)
\right] = T_N \left( Z^{t,\psi} (s) \right) \, ,
$$
holds $\mathbb{P}$-a.s., for a.e. $s \in [t,T] $.

By the dominated convergence theorem,
$$\zeta^N (s,X^{t,\psi}) = \mathbb{E}^{\mathcal{F}_s} T_N \left( Z^{t,\psi} (s) \right) = T_N \left( Z^{t,\psi} (s) \right) \, , $$
\noindent
since $Z^{t,\psi}(s)$ is adapted to $\mathcal{F}_s$.

We can pass to the limit $N \rightarrow \infty$:
$$\lim_{N \rightarrow \infty }\zeta^N (s,\mathbf{X}^{t,\psi}) = Z^{t,\psi} (s) \, ,  $$
a.s., for a.e. $s \in [t,T]$. We define $$
\zeta (s, \phi) = \lim_{N \rightarrow \infty }\zeta^N (s,\phi) \, ,
$$
for every $s, \phi$ such that this limit exists, and $0$ otherwise.

So we have
$$
\zeta (s, \mathbf{X}^{t,\psi}) = Z^{t,\psi} (s) \, ,
$$
a.s., for a.e. $s \in [t,T]$.

From the representation formula \eqref{FK}, we can obtain \eqref{gdg}.
\end{proof}

\section{Mild solution and connections to the Kolmogorov Equation}
\label{sec:mild}

Having characterized both $Y^{t,\psi}$ and $Z^{t,\psi}$ in terms of deterministic functionals of the forward process $X^{t,\psi}$ in the preceding sections, we now turn to the PDE formulation of the problem. In this final section, we give a definition of the mild solution for the PDE \eqref{PDKE}. In particular, we integrate the contribution of the boundary condition into the definition of the mild solution. This extra term is associated with the Stieltjes integral term appearing in the backward dynamic of the system defined in Eq. \eqref{FBSDE}.

\begin{definition}
\label{mild_def}
A continuous function $u:[0,T] \times \Lambda \times \Gamma \rightarrow \mathbb{R}$ is a mild solution to Eq. \eqref{PDKE} if the generalized gradient $\nabla^\sigma u$ is non-empty and the following equality holds true for every $\zeta \in \nabla^\sigma u$, all $t \in [0,T]$ and $\psi \in \Lambda \times \Gamma$:
\begin{equation}
\label{mild}
\displaystyle u(t,\psi) = P_{t,T} \left[ h \right] (\psi)  + \int_t^T P_{t,s} \left[ 	f \left( s, \cdot, u(s, \cdot), \zeta  (s, \cdot), (u(\cdot, \cdot))_s)
\right) \right] (\psi) ds  +
 P_{t,s} \left[ \mathcal{G}
  \right] (\psi),
\end{equation}
where
the operator $\mathcal{G}:\Lambda\times\Gamma\rightarrow\mathbb{R}$ is defined as
\begin{equation}
\label{G_operator}
\mathcal{G}(\psi) = \mathcal{G} (\phi, \varphi) = \int_t^T g [s, \psi, u (s, \psi), (u (\cdot, \psi))_s ] \nu (\phi (s)) d \varphi(s).
\end{equation}
\end{definition}

We recall that the delayed term $(u(\cdot, \psi))_s$ was defined in \eqref{delayed_term}, while the {\it Markov transition semigroup} generated by the operator $\mathcal{L}$ was introduced in Eq. \eqref{PtG}.

We are now in the position to state the existence and uniqueness of a mild solution for Eq. \eqref{PDKE}.

\begin{theorem}[Existence and Uniqueness]
\label{Kolmind}
Let Assumptions $(A_1)-(A_4)$ and $(H_1)-(H_2)$  hold true. Then the function $u$ defined by Eq. \eqref{FK} is the unique mild solution \eqref{mild} to the path-dependent
partial integro-differential Eq. \eqref{PDKE}. Moreover, $u$ is connected to the FBSDE system \eqref{FBSDE} by the following equalities:

\begin{itemize}
    \item $u(s, \mathbf{X}^{t,\psi}) = Y^{t,\psi} (s)$ for every $s \in [0,T]$ and for $\psi \in \Lambda\times\Gamma$; \medskip
    \item $\zeta (s, \mathbf{X}^{t,\psi}) = Z^{t,\psi} (s)$,  for any $\zeta \in \nabla^\sigma u$, $\mathbb{P}$-a.s., for a.e. $s \in [t, T]$. \medskip
\end{itemize}
\end{theorem}

\begin{proof}
\noindent{\it   Existence.}
    We start from the BSDE. We take expectations in both sides of the BSDE to recover the definition of mild solution. Hence, we can rewrite \eqref{mild}
\begin{equation}
\label{mild2}
\begin{array}{l}
\displaystyle u(t,\psi) = \mathbb{E} \left[ h(\mathbf{X}^{t,\psi}) \right]  + \int_t^T \mathbb{E} \left[ 	f \left( s, \mathbf{X}^{t,\psi},  u(s, \mathbf{X}^{t,\psi}), \zeta (s, \mathbf{X}^{t,\psi}), (u(\cdot, \mathbf{X}^{t,\psi}))_s) \,
\right) \right]  ds    \medskip \\
\displaystyle
 \qquad \qquad  \qquad \qquad  \qquad \qquad  +\mathbb{E}  \int_t^T
\left[g
\left(s, \mathbf{X}^{t,\psi} , u(s, \mathbf{X}^{t,\psi}), (u(\cdot, \mathbf{X}^{t,\psi}))_s
 \right)    \right] dA^{t,\psi
}  (s) \, .
 \end{array}
\end{equation}

\noindent
Since
\begin{equation}
\label{reflectedK2}
A^{t,\psi}(s) = \int_{t}^{s}\nu({X}^{t,\psi}\left( r\right) )dK^{t,\psi
} \left( r\right) \, ,
\end{equation}
we have
\begin{equation}
\label{existence1}
\begin{array}{l}
\displaystyle u(t,\psi) = \mathbb{E} \left[ h(\mathbf{X}^{t,\psi}) \right]  + \int_t^T \mathbb{E} \left[ 	f \left( s, \mathbf{X}^{t,\psi}, u(s, \mathbf{X}^{t,\psi}), \zeta (s, \mathbf{X}^{t,\psi}), (u(\cdot, \mathbf{X}^{t,\psi}))_s) \,
\right) \right]  ds    \medskip \\
\displaystyle
 \qquad \qquad  \qquad \qquad  +\mathbb{E}  \int_t^T
\left[g
\left(s, \mathbf{X}^{t,\psi} , u(s, \mathbf{X}^{t,\psi}), (u(\cdot, \mathbf{X}^{t,\psi}))_s
 \right)    \right] \nu({X}^{t,\psi}\left( s\right) )dK^{t,\psi
}  (s) \, ,
 \end{array}
\end{equation}

\noindent
Since the above holds for all  $\varphi$, we have
\begin{equation}
    \label{existence2}
 P_{t,s} \left[\mathcal{G}
  \right] (\phi, \varphi)
= \mathbb{E} \left[ \mathcal{G}  (\mathbf{X}^{t,\psi}  ) \right] \, .
  \end{equation}

By plugging the definition \eqref{PtG} of the semigroup $P_{t,T}$ in \eqref{existence1} and \eqref{existence2}, we arrive at the required equality.

\noindent{\it Uniqueness.}
Let $\tilde{u}$ be a mild solution of \eqref{PDKE} and for any $\zeta \in \nabla^\sigma \tilde{u}$, the following
\begin{equation}
\label{mild_proof}
\displaystyle \tilde{u}(t,\phi, \varphi) = P_{t,T} \left[ h \right] (\phi, \varphi)  + \int_t^T P_{t,s} \left[ 	f \left( s, \cdot, \tilde{u}(s, \cdot), \zeta  (s, \cdot), (\tilde{u}(\cdot, \cdot))_s)
\right) \right] (\phi, \varphi) ds  +
 P_{t,s} \left[ \mathcal{G}
  \right] (\phi, \varphi) \, ,
  \end{equation}
holds. By the Markov property,
\begin{multline}
P_{s,r} \left[ 	f \left( r, \cdot, \tilde{u}(r, \cdot), \zeta  (r, \cdot), (\tilde{u}(\cdot, \cdot))_r)
\right) \right] (\mathbf{X}^{t,\psi})  \\
= \mathbb{E}^{\mathcal{F}_s} \left[  f \left(r, \mathbf{X}^{t,\psi}, \tilde{u}(r, \mathbf{X}^{t,\psi}), \zeta  (r, \mathbf{X}^{t,\psi}), (\tilde{u}(\cdot, \mathbf{X}^{t,\psi}))_r)
\right) \right] \, ,
\end{multline}
and
\[
P_{s,r} \left[ \mathcal{G}
  \right] (\mathbf{X}^{t,\psi})
= \mathbb{E}^{\mathcal{F}_s} \int_s^T g \left(r, \mathbf{X}^{t,\psi},   \tilde{u}(r, \mathbf{X}^{t,\psi}),  (\tilde{u}(\cdot, \mathbf{X}^{t,\psi}))_r \right) \nu ({X}^{t,\psi} (r)) dK^{t,\psi} (r).
\]

Similarly,
$$
P_{s,T} \left[ h \right] (\mathbf{X}^{t,\psi}) = \mathbb{E}^{\mathcal{F}_s} h \left(\mathbf{X}^{t,\psi}\right) \, .
$$

Hence, by evaluating  \eqref{mild_proof} at $t=s$ and $\phi = X^{t,\psi}$ , we get
\begin{equation}
\label{eq_proof}
\begin{array}{l}
\displaystyle
    \tilde{u}(s, \mathbf{X}^{t,\psi}) = \mathbb{E}^{\mathcal{F}_s} h \left(\mathbf{X}^{t,\psi}\right) \, + \int_s^T \mathbb{E}^{\mathcal{F}_s} \left[  f \left( r, \mathbf{X}^{t,\psi},\tilde{u}(s, \mathbf{X}^{t,\psi}), \zeta  (r, \mathbf{X}^{t,\psi}), (\tilde{u}(\cdot, \mathbf{X}^{t,\psi}))_r
\right) \right] dr\medskip\\   \displaystyle \qquad \qquad \qquad \qquad + \mathbb{E}^{\mathcal{F}_s} \int_s^T g \left(r, \mathbf{X}^{t,\psi},  \tilde{u}(r, \mathbf{X}^{t,\psi}),  (\tilde{u}(\cdot, \mathbf{X}^{t,\psi}))_r \right) \nu({X}^{t,\psi} (r)) dK^{t,\psi} (r) \, ,
\end{array}
\end{equation}
with $t \leq s \leq r \leq T$.

Then, we define
\begin{equation}
\begin{array}{c}
\displaystyle \eta =  h \left(\mathbf{X}^{t,\psi}\right) \, +  \int_t^T  \left[  f \left( r, \mathbf{X}^{t,\psi}, \tilde{u}(s, \mathbf{X}^{t,\psi}), \zeta  (r, \mathbf{X}^{t,\psi}), (\tilde{u}(\cdot, \mathbf{X}^{t,\psi}))_r)
\right) \right] dr \\
\displaystyle
 + \int_t^T g \left(r, \mathbf{X}^{t,\psi},  \tilde{u}(r, \mathbf{X}^{t,\psi}),  (\tilde{u}(\cdot, \mathbf{X}^{t,\psi}))_r \right) dA^{t,\psi} (r) \,
\end{array}
\end{equation}
and by Fubini theorem, we have
\begin{equation}
\label{}
    \begin{array}{l}
\displaystyle \tilde{u} (s, \mathbf{X}^{t,\psi}) = \mathbb{E}^{\mathcal{F}_s}[\eta] - \int_t^s f \left( r, \mathbf{X}^{t,\psi}, \tilde{u}(s, \mathbf{X}^{t,\psi}), \zeta  (r, \mathbf{X}^{t,\psi}), (\tilde{u}(\cdot, \mathbf{X}^{t,\psi}))_r)
\right) dr \medskip \\
\displaystyle \qquad \qquad \qquad  - \int_t^s
g \left(r, \mathbf{X}^{t,\psi},  \tilde{u}(r, \mathbf{X}^{t,\psi}),  (\tilde{u}(\cdot, \mathbf{X}^{t,\psi}))_r \right) dA^{t,\psi}(r).
    \end{array}
\end{equation}

Note that $\tilde{u}(t, \phi) = \mathbb{E} [\eta]$ and $\eta$ is $\mathcal{F}_T$-measurable. Hence, by the martingale representation theorem, there
exists a predictable process $\tilde{Z}$ such that $\mathbb{E}^{\mathcal{F}_s} [\eta] = \int_t^s \Tilde{Z}_r dW_r + \tilde{u}(t,\phi)$. We derive that the process $\tilde{u} (s, X^{t,\psi})$
is a continuous semimartingale
with canonical decomposition
\begin{equation}
    \label{canonical}
    \begin{array}{l}
\displaystyle \tilde{u} (s, \mathbf{X}^{t,\psi}) = \int_t^s \Tilde{Z}_r dW_r + \tilde{u} (t,\phi) - \int_t^s f \left( r, \mathbf{X}^{t,\psi}, \tilde{u}(s, \mathbf{X}^{t,\psi}), \zeta  (r, \mathbf{X}^{t,\psi}), (\tilde{u}(\cdot, \mathbf{X}^{t,\psi}))_r)
\right) dr \medskip \\
\displaystyle \qquad \qquad  - \int_t^s
g \left(r, \mathbf{X}^{t,\psi},  \tilde{u}(r, \mathbf{X}^{t,\psi}),  (\tilde{u}(\cdot, \mathbf{X}^{t,\psi}))_r \right) dA^{t,\psi}(r) \, .
    \end{array}
\end{equation}

By computing the quadratic variation, we have
\begin{equation}
    \label{eqproof3}
\langle \tilde{u}(\cdot, \mathbf{X}^{t,\psi}), W (\cdot) \rangle_{[t,s]} = \int_t^s \Tilde{Z}_r dr.
\end{equation}

By plugging \eqref{eqproof3} into \eqref{canonical}, we get
\begin{equation}
    \label{canonical2}
    \begin{array}{l}
\displaystyle \tilde{u} (s, \mathbf{X}^{t,\psi}) = \int_t^s \zeta (r, \mathbf{X}^{t,\psi}) dW_r + \tilde{u} (t,\psi) \medskip \\
\displaystyle \qquad \qquad - \int_t^s f \left( r, \mathbf{X}^{t,\psi}, \tilde{u}(s, \mathbf{X}^{t,\psi}), \zeta  (r, \mathbf{X}^{t,\psi}), (\tilde{u}(\cdot, \mathbf{X}^{t,\psi}))_r)
\right) dr  \medskip \\
\displaystyle \qquad \qquad
- \int_t^s
g \left(r, \mathbf{X}^{t,\psi},  \tilde{u}(r, \mathbf{X}^{t,\psi}),  (\tilde{u}(\cdot, \mathbf{X}^{t,\psi}))_r \right) dA^{t,\psi}(r).
    \end{array}
\end{equation}

Since $\tilde{u}(T, \mathbf{X}^{t,\psi}) = h (\mathbf{X}^{t,\psi})$, we also have
\begin{equation}
    \label{canonical3}
    \begin{array}{l}
\displaystyle u (s, \mathbf{X}^{t,\psi}) + \int_s^T \zeta (r, \mathbf{X}^{t,\psi}) dW_r \medskip \\
\displaystyle \qquad =  h (\mathbf{X}^{t,\psi}) - \int_s^T f \left( r, \mathbf{X}^{t,\psi}, \tilde{u}(s, \mathbf{X}^{t,\psi}), \zeta  (r, \mathbf{X}^{t,\psi}), (\tilde{u}(\cdot, \mathbf{X}^{t,\psi}))_r)
\right) dr  \medskip \\
\displaystyle \qquad \quad
- \int_s^T
g \left(r, \mathbf{X}^{t,\psi},  \tilde{u}(r, \mathbf{X}^{t,\psi}),  (\tilde{u}(\cdot, \mathbf{X}^{t,\psi}))_r \right) dA^{t,\psi}(r).
    \end{array}
\end{equation}

By comparison with BSDE in \eqref{FBSDE}, the pairs $\left(Y^{t,\psi} \, , \,  Z^{t,\psi}\right)$ and $\left(\tilde{u} (s, \mathbf{X}^{t,\psi}) \, , \zeta (r, \mathbf{X}^{t,\psi}) \, \right )$ solve the same equation, which is unique by Prop.~\ref{continuous_dep}.
\end{proof}

In the next paragraph, we delve into the penalization scheme for justifying the directional gradient and the infinitesimal generator for motivating the link between the Kolmogorov Equation and the mild solution defined in terms of a semigroup.

\subsection{The convergence of the penalized PDE via penalization scheme}
\label{subsec:penalization}

To rigorously justify the use of the generalized directional gradient in the definition of mild solution, we employ a penalization scheme. This approach approximates the reflected process by a sequence of unconstrained SDEs with a penalizing drift term.

Following a classical penalization scheme approach, see e.g. \cite{bouf}, we suppose that $b$ and $\sigma$ are defined on $[0,T] \times \bar{\Lambda}$, where $\bar{\Lambda}$ is $\mathcal{C} ([0,T]; \mathbb{R}^d)$. We introduce auxiliary functions to define the penalization term. Let $\rho \in C^1_b(\mathbb{R}^d)$ be a function such that $\rho = 0$ in $\bar{G}$, $\rho > 0$ in $\mathbb{R}^d \setminus \bar{G}$, and $\rho(x) = (d(x,\partial G))^2$ in a neighborhood of $\partial G$. On the other hand, since the domain $G$ is smooth (say $C^3$), it is possible to consider an extension $\ell \in C^2_b(\mathbb{R}^d)$ of the function $d(\cdot, \partial G)$ defined on the restriction to $G$ of a neighborhood of $\partial G$ such that $G$ and $\partial G$ are characterized by
\begin{equation}
\label{domain_char}
G = \{x \in \mathbb{R}^d : \ell(x) < 0\} \quad \text{and} \quad \partial G = \{x \in \mathbb{R}^d : \ell(x) = 0\}
\end{equation}
and for all $x \in \partial G$, $\nabla \ell(x)$ points outward from $D$ (see for example \cite{Sznitman}, Remark 3.1). Following \cite{bouf}, we define the penalization term as
\begin{equation}
\label{penalization_term}
\delta(x) := \nabla \rho(x) \, .
\end{equation}
In particular, we may and do choose $\rho$ and $\ell$ such that
\begin{equation}
\label{penalization_property}
\langle \nabla \ell(x), \delta(x) \rangle \leq 0, \quad \text{for all } x \in \mathbb{R}^d \, .
\end{equation}

The approximating SDE reads
\begin{multline}
\label{approx_for1}
X^{n,t,\phi }\left( s\right) =\phi \left( t\right)
+\int_{t}^{s}b(r,X^{n,t,\phi })dr+\int_{t}^{s}\sigma (r,X^{n,t,\phi
})dW\left( r\right) \\
-n\int_{t}^{s} \delta(X^{n,t,\phi }\left( r\right)) dr.
\end{multline}

For $\psi:=(\phi, \varphi)$, by denoting $ \displaystyle K^{n,t,\psi }(s):= \varphi(t) - n\int_{t}^{s}\delta(X^{n,t,\phi
}\left( r\right)) dr$,
we have that%
\begin{equation}
\label{conv2}
(X^{n,t,\phi },K^{n,t,\psi })\rightarrow ({X}^{t,\psi}, {K}^{t,\psi}),
\end{equation}
where the convergence holds almost surely on $\bar{\Lambda}\times\bar{\Lambda}$, if $\psi\in\Lambda\times\Gamma$, as claimed in \cite{Sznitman}.

\smallskip

We can rewrite Eq. \eqref{approx_for1} as
\begin{equation}
\label{approx_for2}
X^{n,t,\phi }\left( s\right) =\phi \left( t\right)
+\int_{t}^{s}{b}^n(r,X^{n,t,\phi })dr+\int_{t}^{s}\sigma (r,X^{n,t,\phi
})dW\left( r\right),
\end{equation}
by defining ${b}^n: [0,T] \times \bar{\Lambda} \rightarrow \mathbb{R}^d$ as
\begin{equation}
    \label{tildeb}
    {b}^n\left(r,\phi\right) =  b \left(r,\phi\right)
 - n \delta\left(\phi\left( r\right) \right).
\end{equation}

\medskip

Plugging Eq. \eqref{approx_for1} into
the BSDE of the system \eqref{FBSDE} we get the following sequence of approximating BSDEs, supposing again that the coefficients $h$, $f$ and $g$ are depending on $\phi\in\bar{\Lambda}$ (instead $\phi\in\Lambda$):
\begin{equation}
\label{backward}
    \begin{array}{l}
\displaystyle Y^{n, t,\psi }\left( s\right) =h(X^{n, t,\phi
},K^{n, t,\psi
})+\int_{s}^{T}f(r,X^{n,t,\phi },K^{n, t,\psi
},Y^{n, t,\psi }\left( r\right) ,Z^{n, t,\psi }\left(
r\right) ,Y_{r}^{n, t,\psi })dr \smallskip  \\
\quad \displaystyle -n \int_{s}^{T}g(r,X^{n,t,\phi },K^{n,t,\psi },Y^{n, t,\psi
}\left( r\right) ,Y_{r}^{n, t,\psi }) \left\langle \nabla \ell(X^{n,t,\phi }\left( r\right)), \delta(X^{n,t,\phi }\left( r\right) )\right\rangle dr  \smallskip \\
\quad \displaystyle -\int_{s}^{T}Z^{n, t,\psi
}\left( r\right) dW\left( r\right) , \medskip
    \end{array}
\end{equation}
that we can rewrite in compact form
\begin{multline}
\label{backward_tilde}
\displaystyle Y^{n, t,\psi }\left( s\right) =h(X^{n, t,\phi
},K^{n,t,\psi })+\int_{s}^{T}{f}^n\left( r,X^{n, t,\phi },K^{n,t,\psi },Y^{n, t,\psi }\left( r\right) ,Z^{n, t,\psi }\left(
r\right) ,Y_{r}^{n, t,\psi }\right) dr  \\
-\int_{s}^{T}Z^{n, t,\psi
}\left( r\right) dW\left( r\right) , \medskip
\end{multline}
by introducing the coefficient ${f}^n$ by
\begin{equation}
\label{tildef}
{f}^n\left( r,\psi,y ,z ,\hat{y}\right):=  f\left( r,\psi,y ,z ,\hat{y}\right)
-n  g(r,\psi,y ,\hat{y})  \left\langle \nabla \ell\left( \phi\left( r\right) \right), \delta\left( \phi\left( r\right) \right)\right\rangle \, .
\end{equation}

\medskip

To show that the corresponding solution $(Y^{n,t,\psi },Z^{n,t,\psi
})$ of the approximating BSDE converges to $(Y^{t,\psi},Z^{t,\psi})$, we
can use the convergence result given in \cite{spa}. The coefficients satisfy the assumptions required for Theorem 7
in \cite{spa}. Indeed, since $X^{n,t,\phi}$ converges to $X^{t,\psi}$ and the
penalization term $-n\langle \nabla \ell(X^{n,t,\phi}(r)), \delta(X^{n,t,\phi}(r)) \rangle$ concentrates near
the boundary (as shown in Remark 2.2 of \cite{bouf}), the products remain
uniformly bounded. The detailed verification follows the arguments in
\cite{bouf}.

Hence, by \cite[Theorem 7]{spa}, we obtain that
\begin{equation}
\label{conv3}(Y^{n,t,\psi },Z^{n,t,\psi })\rightarrow (Y^{t,\psi},Z^{t,\psi})
\end{equation}
almost surely on $\mathcal{C}([0,T])\times L^2([0,T];\mathbb{R}^{d\times d'})$, if $\psi\in\Lambda\times\Gamma$.

We first remark that in \cite{CDMZ}, in the case that the coefficients $h$, $f$ and $g$ don't depend on the parameter $\varphi\in\Gamma$, it was shown that the solution of the approximating FSBDE system is related to the (probabilistic, in some sense) solution of the following PDE:
\begin{equation}
\label{penalizedpde2}
\left\{
\begin{array}{l}
-\partial _{t}u^{n}(t,\phi )-{\mathcal{L}^{n}} u^{n}(t,\phi )-{f}^n\left(t,\phi
,u^{n}(t,\phi ),\partial _{x}u^{n}\left( t,\phi \right) \sigma (t,\phi
),\left( u^{n}(\cdot ,\phi )\right) _{t}\right) =0,\medskip  \\
u^{n}(T,\phi )=h(\phi ),\qquad \phi\in\bar{\Lambda},
\end{array}%
\right.
\end{equation}
where
\begin{equation*}
\mathcal{L}^{n} u (t,\phi) := \frac{1}{2} \operatorname{Tr} \left[  \sigma (t,\phi)\sigma^{\ast} (t,\phi )\partial_{xx}^{2}
u(t,\phi ) \right] + \langle b^{n}(t,\phi) ,\partial_{x}u(t,\phi )\rangle.
\end{equation*}

To state the connection with the approximating FBSDE system, we define the notion of a mild solution $u^n$ to the above PDE written in the general case:
\begin{equation}
\label{penalizedPDE}
\left\{
\begin{array}{l}
-\partial _{t}u^{n}(t,\psi )-\mathcal{L}u^{n}(t,\psi )+n \delta\left(\phi\left( r\right)\right)\partial _{x}u^{n}\left( t,\psi \right)\\
\qquad \qquad-f(t,\psi
,u^{n}(t,\psi ),\partial _{x}u^{n}\left( t,\psi \right) \sigma (t,\phi
),\left( u^{n}(\cdot ,\psi )\right) _{t}) \\
\qquad \qquad +n\left\langle \nabla \ell(\phi(t)), \delta(\phi(t))\right\rangle g\left(t,\psi ,u^{n}(t,\phi
),\left( u^{n}(\cdot ,\phi )\right) _{t}\right)=0,\medskip  \\
u^{n}(T,\psi )=h(\psi ),\qquad \psi=(\phi,\varphi)\in\bar{\Lambda}\times\Gamma.%
\end{array}%
\right.
\end{equation}
In this regard, we introduce the approximating semigroup $P^n$ as
\begin{equation}
\label{semigroup_n}
P^n_{t,s} \left[f\right]  (\phi,\varphi):=\mathbb{E}\left[  f\left(  X^{n,t,\phi
}(\cdot\wedge s),K^{n,t,(\phi,\varphi)}(\cdot\wedge s)\right)  \right]
\,,\qquad f\in\mathcal{B}_{b}\left( \bar{\Lambda}
\times \Gamma  \right) \, .
\end{equation}

By \eqref{conv2}, we have $P_{t,s}^{n} [f](\psi)\rightarrow P_{t,s}[f_{\vert\Lambda\times\Gamma}](\psi)$ for any $f \in \mathcal{C}_b \left( \bar{\Lambda}
\times \Gamma \right)$ and $\psi\in\Lambda\times\Gamma$, by the Lebesgue dominated convergence theorem, recalling that $P_{t,s}$ is the semigroup defined in Eq. \eqref{PtG}.

\begin{definition}
\label{mild_def_n}
A continuous function $u^n:[0,T] \times (\bar{\Lambda}\times\Gamma)\rightarrow \mathbb{R}$ is a mild solution to Eq. \eqref{penalizedpde2} if there exists $C>0$ and $m\geq 0$ such that for any $t \in [0,T]$ and any $\psi_1, \psi_2 \in \bar{\Lambda}\times\Gamma$, $u^n$ satisfies
\begin{equation}
\begin{array}{l}
| u^n(t, \psi_1) - u^n (t, \psi_2) | \leq C || \psi_1 - \psi_2|| (1 + || \psi_1|| + ||\psi_2|| )^m \, ; \smallskip
\\
\displaystyle  |u^n (t,0,0)| \leq C \, ;
\end{array}
\label{mild1}
\end{equation}%
and the following equality holds true:
\begin{equation}
\begin{array}{cc}
\label{mild_n}
\displaystyle u^n(t,\psi) = P^n_{t,T} h(\psi)  + \int_t^T P^n_{t,s} \left[ 	{f}^n \left( s, \cdot, u(s, \cdot), \partial_x u^n (s, \cdot) \sigma(s, \cdot), (u^n(\cdot, \cdot))_s)
\right)  \right] (\psi) ds \, ,
\end{array}
\end{equation}
for all $t \in [0,T]$ and $\psi \in \bar{\Lambda}\times\Gamma$, where $\partial_x u^n (s, \cdot)$ denotes the Clarke's gradient of $u^n (s, \cdot)$.
\end{definition}

Of course, we can write \eqref{mild_n} in a similar way as in Definition \ref{mild_def}, by defining the function $\mathcal{G}^n:\bar{\Lambda}\times\Gamma\rightarrow\mathbb{R}$ as
\begin{equation}
\label{G_operator2}
\mathcal{G}^n (\phi, \varphi) = -\int_t^T g [s, \phi, u (s, \phi), (u (\cdot, \phi))_s ] \langle \nabla \ell(\phi(s)), \delta(\phi(s)) \rangle d \varphi(s) \, ,
\end{equation}
\noindent
and express the mild solution \eqref{mild_n} as
\begin{equation}
\begin{array}{l}
\label{semigroup2}
\displaystyle u^n(t,\psi) = P^n_{t,T} h(\psi)
+
\int_t^T P^n_{t,s} \Big[ 	f \left( s, \cdot, u(s, \cdot), \partial_x u (s, \cdot) \sigma(s, \cdot), (u(\cdot, \cdot))_s)
\right) \Big] (\psi) ds \medskip \\
\displaystyle \qquad \qquad +  P^n_{t,s} \left[ \mathcal{G}^n
\right] (\psi) \, .
\end{array}
\end{equation}
Moreover, for the FSBDE system described in Eqs. \eqref{approx_for1}, \eqref{backward} we know that the following formulae
\begin{align*}
&Y^{n,t,\psi }(s)=u^{n}(s,X^{n,t,\phi },K^{n,t,\psi }),\ \forall t,s\in \lbrack 0,T] \,\\
&Z^{n,t,\psi }(s) = \zeta^n (s, X^{n,t,\phi},K^{n,t,\psi } ) \in\nabla^{\sigma}u^{n}(s,X^{n,t,\phi },K^{n,t,\psi }) \, , \qquad \textit{
 for a.e. } s \in [t,T] \, ,
\end{align*}
hold (the proof is very similar to that of Theorem \ref{Kolmind}, the difference is given by the fact that we have no reflection ($G=\mathbb{R}^n$)), where $\nabla^{\sigma}$ defines the generalized directional gradient introduced in Eq.  \eqref{gdg}.

\begin{remark}
   We remark that $u^n$ is locally Lipschitz, so its Clarke gradient is well defined. Moreover, $\partial_x u^n\sigma\in\nabla^{\sigma}u^n$. On the other hand, $u$ is proven to be just continuous. Hence, given its lack of regularity, in order to define its generalized gradient, we cannot rely on the notion of Clarke. Therefore, we chose to proceed as in \cite{fuhrman_tessitore} or \cite{russovallois}, characterizing $Z^{t,\psi}$ as the generalized gradient through the stochastic definition given in Eq. \eqref{jqv}. But this choice is now justified by the convergence of the sequences $(u_n)$ and $(\partial_x u^n\sigma)$ to $u$, respectively $\nabla^\sigma u$, by \eqref{conv3}.
\end{remark}

\subsection{The infinitesimal generator}
\label{subsec:inf}

In this final section, we connect the infinitesimal generator of the transition semigroup to the differential operator of the Kolmogorov equation \eqref{PDKE}. To motivate this connection, we introduce some basic concepts from non-anticipative calculus. Although the general theory of functional It\^o calculus is usually developed on the space of c\`adl\`ag paths to accommodate jumps, we note that in our setting, all processes are continuous.

For $m\in\mathbb{N}^{\ast}$, let $\hat{\Lambda}$ be the space of
$2d$-dimensional c\`{a}dl\`{a}g functions defined on $[0,T]$. For a
non-anticipative function $F:\hat{\Lambda}\rightarrow\mathbb{R}$,
a first-order partial vertical derivative (i.e., with respect to the spatial
variable) at $\phi$ is defined as
\begin{equation}
\partial_{x_{i}}F\left(\phi\right)  =\lim_{h\rightarrow0}\frac{F\left(
\phi+h\mathds{1}_{[0,T]}e_{i}\right)  -F\left(\phi\right)  }%
{h}\,,\ 1\leq i\leq2d,\label{vertical}%
\end{equation}
whenever the limit exists, where $\{e_{i}\}_{i\in\{1,\ldots,2d\}}$ is the
standard basis of $\mathbb{R}^{2d}$. The second-order vertical derivatives $\partial
_{x_{i}x_{j}}F$, with $1\leq i,j\leq2d$ are defined as usual, as vertical
derivatives of the first-order vertical derivatives. By $\mathcal{C}_{b}%
^{1,2}(\hat{\Lambda})$ we denote the set of all non-anticipative
bounded functions such that their first and second-order vertical derivatives
exist and are bounded. For the
classical non-anticipative setting, see \cite{cont}, \cite{cont2}
and \cite{dupire} for more details.

Now, we denote $\mathcal{C}_{b}^{2}(
(\Lambda\times\Gamma))$ the space of all non-anticipative functions $F:(\Lambda\times\Gamma)\rightarrow\mathbb{R}$ such that there exists an extension $\hat{F}:\hat{\Lambda}\rightarrow\mathbb{R}$ of $F$ such
that $\hat{F}\in\mathcal{C}_{b}^{2}(\hat{\Lambda})$. In that
case, for $\psi\in(\Lambda\times\Gamma)$ and $1\leq
i,j\leq2d$, we define $\partial_{x_{i}}F\left(\phi\right)  $,
$\partial_{x_{i}x_{j}}F\left(\phi\right)  $ as $\partial_{x_{i}}\hat{F}\left(\phi\right)  $,
$\partial_{x_{i}x_{j}}\hat{F}\left(\phi\right)  $, respectively. It can be shown that these
definitions don't depend on the choice of the extension $\hat{F}$.

The generator of the transition semigroup $(P_{t,s})_{0 \leq t \leq s \leq T}$ defined in Eq. \eqref{PtG} is a key tool for understanding the PDE \eqref{PDKE}.

\begin{definition}
For a continuous function $f: \Lambda \times \Gamma \rightarrow \mathbb{R}$, the infinitesimal generator $\mathcal{L}_t$ is defined by
\begin{equation}
\label{generator}
\mathcal{L}_t[f] = \lim_{s \rightarrow t^+} \frac{P_{t,s}[f]- f}{s - t} \, ,
\end{equation}
whenever the limit exists in $\mathcal{C} (\Lambda \times \Gamma)$.
\end{definition}

The following result characterises $\mathcal{L}_t$ explicitly for sufficiently regular functions satisfying the Neumann boundary condition.

\begin{proposition}[Characterization of the infinitesimal generator]
\label{prop:generator}
Let Assumptions $(A_1)$-$(A_4)$ hold. Then, for functions $f \in \mathcal{C}^2(\Lambda \times \Gamma)$ satisfying the homogeneous Neumann boundary condition $\frac{\partial f}{\partial \nu}(\phi, 0) = 0$ for all $\phi \in \Lambda$ with $\phi(t) \in \partial G$, the infinitesimal generator  $\mathcal{L}_t$ defined in Eq. \eqref{generator} of the semigroup $(P_{t,s})_{0 \leq t \leq s \leq T}$ is given by
\begin{equation}
\label{operator2}
\mathcal{L}_t f(\psi)  = \mathcal{L} f (t,\psi) = \frac{1}{2} \operatorname{Tr} \left[ \sigma(t,\phi)\sigma^{\ast}(t,\phi)\partial_{xx}^{2}f(\psi) \right] + \langle b(t,\phi), \partial_{x}f(\psi)\rangle \, ,
\end{equation}
for all $\psi = (\phi, \varphi) \in \Lambda \times \Gamma$.
\end{proposition}

\begin{proof}
Given $f \in \mathcal{C}^2(\Lambda \times \Gamma)$, we analyze $f(X^{t,\psi}(s))$ using the non-anticipative version of It\^o's formula. For simplicity of notation, we denote $\mathbf{X}^{t,\psi} = (X^{t,\phi}, K^{t,\varphi})$.

\begin{equation}
\begin{array}{l}
\displaystyle f \left( \mathbf{X}^{t,\psi}(\cdot\wedge s)\right) = f(\psi) + \int_t^s \nabla_x f \left( \mathbf{X}^{t,\psi}(\cdot\wedge r)\right) \cdot d \mathbf{X}^{t,\psi}(r) + \smallskip \\
\displaystyle \qquad + \int_t^s \frac{1}{2} \text{tr} \left( \nabla_x^2 f \left( \mathbf{X}^{t,\psi}(\cdot\wedge r)\right) \right) d [\mathbf{X}^{t,\psi}](r) \smallskip \, ,
\end{array}
\end{equation}
\noindent
where $\nabla_x$ is the vertical derivative defined in Eq. \eqref{vertical}, and $\nabla_x^2$ is the corresponding Hessian matrix.

Substituting the dynamics from Eq. \eqref{forward2}, and writing component-wise $f(\mathbf{X}^{t,\psi}) = (f_X(X^{t,\phi}), f_K(K^{t,\varphi}))$ and $\nabla_x f = (\nabla_X f, \nabla_K f)$, we obtain

\begin{equation}
\begin{array}{l}
\displaystyle f_X \left( \mathbf{X}^{t,\psi}(\cdot\wedge s)\right) = f_X(\psi) + \int_t^s \nabla_X f_X \left( \mathbf{X}^{t,\psi}(\cdot \wedge r)\right) b(r, X^{t,\phi}) dr + \smallskip\\
\displaystyle \qquad + \int_t^s \nabla_X f_X \left( \mathbf{X}^{t,\psi}(\cdot \wedge r)\right) \sigma(r, X^{t,\phi}) dW(r) + \smallskip\\
\displaystyle \qquad + \int_t^s \nabla_X f_X \left( \mathbf{X}^{t,\psi}(\cdot \wedge r)\right) dK^{t,\varphi}(r) + \smallskip\\
\displaystyle \qquad + \int_t^s \nabla_K f_X \left( \mathbf{X}^{t,\psi}\right) d(K^{t,\varphi})(r) + \smallskip\\
\displaystyle \qquad + \int_t^s \frac{1}{2} \nabla_X^2 f_X \left( \mathbf{X}^{t,\psi}(\cdot\wedge r)\right) d [X^{t,\phi},X^{t,\phi}](r) \smallskip \, .
\end{array}
\end{equation}

\noindent
For functions that satisfy the Neumann boundary condition, we have
\begin{equation}
\int_t^s \nabla_X f_X \left( \mathbf{X}^{t,\psi}(\cdot \wedge r)\right) \nu(X^{t,\phi}) dA^{t,\phi}(r) = 0 \, ,
\end{equation}
and
\begin{equation}
\int_t^s \nabla_K f_X \left( \mathbf{X}^{t,\psi}\right) \nu(X^{t,\phi}) d(A^{t,\phi})(r) = 0 \, .
\end{equation}

Taking the expectation, dividing by $(s-t)$, and letting $s \to t^+$, we obtain
\begin{equation}
\label{generator2}
\mathcal{L}_t f(\psi) = \frac{1}{2} \text{Tr} \left[ \sigma(t,\phi)\sigma^{\ast}(t,\phi)\partial_{xx}^{2} f(\psi) \right] + \langle b(t,\phi), \partial_{x}f(\psi)\rangle \, .
\end{equation}

This establishes the explicit form of the infinitesimal generator $L_t$ for functions in $\mathcal{C}^2(\Lambda \times \Gamma)$ satisfying the boundary conditions.
\end{proof}

This proposition shows that, on smooth functions, the generator coincides with the differential operator $\mathcal{L}$ defined in Eq. \eqref{operator}. For less regular functions, the mild solution concept and the generalized gradient allow us to extend this connection rigorously, as shown in Theorem \ref{Kolmind}.

\bibliographystyle{unsrt}

\end{document}